\documentclass[a4paper]{amsart}
\usepackage[english]{babel}
\usepackage[utf8]{inputenc} 
\usepackage{amssymb,amsthm,amsmath} 
\usepackage{indentfirst}
\usepackage[makeroom]{cancel}
\usepackage{microtype}
\usepackage{color}
\usepackage{enumerate} 
\usepackage{tikz} 
\usetikzlibrary{cd} 
\usepackage{bbm} 
\usepackage{mathdots}
\usepackage[hidelinks,draft=false]{hyperref}
\usepackage[colorinlistoftodos,textsize=small,textwidth=80]{todonotes}
\usepackage[toc]{appendix}
\usepackage{mathtools}
\usepackage{cite}
\usepackage{stmaryrd} 
\usepackage[left=3cm, right=3cm, top=3cm, bottom=3cm]{geometry}
\usepackage{mathabx,epsfig}

\usepackage[all]{xy}

\hypersetup{
    linktoc=all,     
}

\makeatletter
\DeclareFontEncoding{LS1}{}{}
\DeclareFontSubstitution{LS1}{stix}{m}{n}
\DeclareMathAlphabet{\mathscr}{LS1}{stixscr}{m}{n}
\makeatother

\usetikzlibrary{arrows}

\usepackage{tikz}
\usepackage{pgfplots}
\pgfplotsset{compat=newest}
\usetikzlibrary{decorations.markings}

\usepackage{soul}

\newcommand{\Pic}{\operatorname{Pic}}

\newcommand{\K}{\mathbb{K}}

\newcommand{\Vect}{\operatorname{Vect}}

\newtheorem{theorem}{Theorem}[section]
\newtheorem{corollary}[theorem]{Corollary}
\newtheorem{lemma}[theorem]{Lemma}
\newtheorem{proposition}[theorem]{Proposition}
\theoremstyle{definition}
\newtheorem{definition}[theorem]{Definition}
\newtheorem{example}[theorem]{Example}
\newtheorem{remark}[theorem]{Remark}

\numberwithin{equation}{section}

\title{Projective and anomalous representations of categories and their linearizations}

\author{Domenico Fiorenza}
\address{Sapienza Universit\`a di Roma; Dipartimento di Matematica ``Guido Castelnuovo'', P.le Aldo Moro, 5 - 00185 - Roma, Italy; 
}
\email{fiorenza@mat.uniroma1.it}
\author{Chetan Vuppulury}
\address{Sapienza Universit\`a di Roma; Dipartimento di Matematica ``Guido Castelnuovo'', P.le Aldo Moro, 5 - 00185 - Roma, Italy; }
\email{chetan.vuppulury@uniroma1.it}

\begin{document}

\begin{abstract}
    We invesigate the relation between projective and anomalous representations of categories, and show how to any anomaly $J\colon \mathcal{C}\to 2\mathrm{Vect}$ one can associate an extension $\mathcal{C}^J$ of $\mathcal{C}$ and a subcategory $\mathcal{C}^J_{\mathrm{ST}}$ of $\mathcal{C}^J$ with the property that: (i) anomalous representations of $\mathcal{C}$ with anomaly $J$ are equivalent to $\mathrm{Vect}$-linear functors $E\colon \mathcal{C}^J\to \mathrm{Vect}$, and (ii) these are in turn equivalent to linear representations of $\mathcal{C}^J_{\mathrm{ST}}$ where ``$J$ acts as scalars''. This construction, inspired by and generalizing the technique used to linearize anomalous functorial field theories in the physics literature, can be seen as a multi-object version of the classical relation between projective representations of a group $G$, with given $2$-cocycle $\alpha$, and linear representations of the central extension $G^\alpha$ of $G$ associated with $\alpha$.    
\end{abstract}

\maketitle

\setcounter{tocdepth}{1}
\tableofcontents

\section{Introduction}

{A functorial field theory is a (smooth) symmetric monoidal functor from a category of bordisms endowed with geometric structure (e.g., a Riemannian metric) to a linear category, such as finite-dimensional vector spaces or Hilbert spaces. This perspective, originating in the work of Atiyah and Segal \cite{atiyah, segal}, has become a fundamental organizing principle in quantum field theory. For instance, a 2-dimensional conformal field theory with values in super (i.e.\ $\mathbb{Z}/2$-graded) Hilbert spaces can be described as a symmetric monoidal functor
\[
Z\colon \mathrm{Bord}_2^{\mathrm{conf}}\to \mathrm{sHilb},
\]
where $\mathrm{Bord}_2^{\mathrm{conf}}$ denotes the category of 1-dimensional spin manifolds and conformal spin bordisms.
A recurring phenomenon in the construction of such theories is a systematic but controlled failure of strict functoriality, commonly referred to as an \emph{anomaly} in the physics literature. This phenomenon has been extensively studied from both physical and mathematical perspectives; see, for example, the recent overview by Freed in \cite{freed-anomalies} as well as the treatment in \cite{Freed2014, freed-teleman}. A typical example is the fermionic anomaly of conformal spin field theories \cite{ludewig-roos}. 

One way to accommodate anomalies is to enlarge the bordism category so that the additional data compensates for the failure of functoriality. For instance, in \cite{what-is-an-elliptic-object}, Stolz and Teichner define a Clifford linear conformal field theory of degree $n\in \mathbb{Z}$, as a continuous functor
\[
E \colon  \mathrm{CliffordBord}_{2;n}^{\mathrm{conf}}\to \mathrm{sHilb},
\]
where $ \mathrm{CliffordBord}_{2;n}^{\mathrm{conf}}$ is a category whose objects are the same as $\mathrm{Bord}_2^{\mathrm{conf}}$, but whose morphisms are pairs $(\Sigma, \Psi)$, where $\Sigma$ is a conformal spin bordism, and $\Psi$ is an element in the $n$-th tensor power of the fermionic anomaly line. While this construction may appear somewhat ad hoc at first sight, it reflects a more general and conceptually natural mechanism.

From a higher-categorical viewpoint, anomalous functoriality can be encoded in several equivalent ways. On the one hand, one may consider functors valued in a projectivization of a linear category, in which scalar automorphisms are promoted to elements in the Picard 2-group or its higher versions. On the other hand, anomalous theories can be described in terms of cocycle data valued in suitable higher groupoids, together with coherence data controlling the failure of strict functoriality. A third perspective realizes anomalous theories as honest linear functors defined on suitable extensions of the source category classified by such cocycles. Variants of these constructions appear throughout the literature on higher category theory, projective representations, and extended field theories; in addition to the aforementioned references, see, for example, \cite{lurie-cobordism, stolz-teichner-susy, bartlett-douglas-schommer-pries-vicary, fiorenza-valentino, johnson-freyd-scheimbauer, scheimbauer-stempfhuber}.

These different viewpoints are closely related to well-known categorical constructions. In particular, the passage from projectivized targets to cocycle data and vice versa can be understood in terms of universal properties of categorical quotients, while the relationship between cocycle-adapted coherence data and extensions is expressed through a Grothendieck construction, and so it is a version of the straightening/unstraightening correspondence in higher category theory \cite{htt}. Despite their ubiquity, these perspectives are often developed in different contexts and with different conventions, and their precise relationship is not always made explicit.

The aim of the present article is to provide a systematic and self-contained comparison of these approaches in the setting of higher categories. More precisely, given a functor
\[
J\colon \mathcal{C}\to 2\mathrm{Vect},
\]
where $\mathcal{C}$ is an $\infty$-category and $2\mathrm{Vect}$ denotes the Morita 2-category of finite-dimensional algebras, bimodules, and intertwiners, we introduce a notion of $\mathrm{Vect}$-valued anomalous representations of $\mathcal{C}$ with anomaly $J$. We show that these are equivalent to $\mathrm{Vect}$-linear functors defined on a canonical extension $\mathcal{C}^J$ of $\mathcal{C}$, which carries a natural structure of module category over $\mathrm{Vect}$. 

Moreover, $\mathcal{C}^J$ contains a distinguished subcategory $\mathcal{C}^J_{\mathrm{ST}}$, with the same objects as $\mathcal{C}$, such that $\mathrm{Vect}$-linear functors from $\mathcal{C}^J$ to $\mathrm{Vect}$ correspond to linear representations of $\mathcal{C}^J_{\mathrm{ST}}$ in which the anomaly $J$ acts by scalars. This result can be regarded as a multi-object  analogue of the Eilenberg--Watts theorem on additive and cocontinuous functors
between categories of modules \cite{eilenberg, watts}, and is closely related to the classical correspondence between projective representations of a group and linear representations of its central extensions. In particular, when the anomaly $J$ is invertible and suitably structured, anomalous representations reduce to projective representations. In this sense, the framework developed here provides a natural higher-categorical generalization of projective representation theory.

The constructions considered in this paper admit further generalizations, for instance to super vector spaces and to analytic settings involving (super) Hilbert spaces and (super) 2-Hilbert spaces, modelled as the 2-category of (super) von Neumann algebras, Hilbert bimodules and continuous intertwiners. \cite{schmidpeter}. While we do not pursue these analytic aspects in detail here, the formalism developed in the present work applies directly to such settings. In particular, when applied to the $n$-th tensor power of the fermionic anomaly
\[
J_{\mathrm{Fermionic}}\colon \mathrm{Bord}_2^{\mathrm{conf}}\to \mathrm{s2Hilb},
\]
our construction recovers the degree $n$ Clifford-linear conformal field theories  of Stolz and Teichner as instances of a general universal procedure.
}
\vskip 1cm

This article is based on C.V. PhD Thesis \cite{vuppulury}. We thank Matthias Ludewig, Christoph Schweigert, Jim Stasheff, {and the referee} for very useful comments and suggestions. D.F.  was partially supported by 2023 Sapienza research grant ``Representation Theory and Applications", by 2024 Sapienza research grant ``Global, local and infinitesimal aspects of moduli spaces'', and by PRIN 2022 ``Moduli spaces and special varieties'' CUP B53D23009140006. D.F. is a member of the Gruppo Nazionale per le Strutture Algebriche, Geometriche e le loro Applicazioni.
(GNSAGA-INdAM).

\section{Notation and conventions}
Throughout the whole paper, $\K$ will denote a field. By $\Vect_\K$ we denote the symmetric monoidal category of finite dimensional vector spaces over $\K$. For $A$ and $B$ two $\K$-algebras, ${}_{A}\mathbf{Mod}_{B}$ will denote the abelian category of $(A,B)$-bimodules that are finite dimensional over $\K$. By $2\mathrm{Vect}_\K$ we denote the Morita 2-category of finite dimensional $\K$-algebras, bimodules, and intertwiners. That is, $2\mathrm{Vect}_\K$ is the 2-category whose objects are finite dimensional $\K$-algebras, whose 1-morphisms are finite dimensional (over $\K$) bimodules, and whose 2-morphisms are morphisms of bimodules:
\[
2\mathrm{Vect}_\K(A_0,A_1)={}_{A_1}\mathbf{Mod}_{A_0}.
\]
The composition of 1-morphisms is given by tensor products:
\[
2\mathrm{Vect}_\K(A_1,A_2)\times 2\mathrm{Vect}_\K(A_0,A_1)\xrightarrow{\circ} 2\mathrm{Vect}_\K(A_0,A_2)
\]
is given by
\begin{align*}
{}_{A_2}\mathbf{Mod}_{A_1}\times {}_{A_1}\mathbf{Mod}_{A_0}&\to {}_{A_2}\mathbf{Mod}_{A_0}\\
(M_{21},M_{10})&\mapsto M_{21}\otimes_{A_1}M_{10}.
\end{align*}
The identity 1-morphisms are algebras sees as bimodules over themselves. Tensor product of algebras over $\K$ endows $2\mathrm{Vect}_\K$ with a natural structure of symmetric monoidal 2-category, with unit object the $\K$-algebra $\K$; see, e.g.\cite{shulman-sm2c}. We will take $2\mathrm{Vect}_\K$ as our preferred model for the symmetric monoidal 2-category of finite dimensional 2-vector spaces over $\K$, hence the notation. 
\par
We will be constantly identifying a (higher) category with its nerve. {\color{black}Indeed, at least for low values of $n$, an $(\infty,n)$-category can be explicitly described as a simplicial set with a few special properties. For instance, $(\infty,0)$ categories (or, equivalently, $\infty$-groupoids) can be defined as Kan complexes (i.e., simplicial sets for which {all the  horns} are required to have a filler) and $(\infty,1)$-categories as weak Kan complexes (i.e., simplicial sets for which only the internal horns are required to have a filler). The simplicial characterization of $(\infty,2)$-categories appears to be a little less widely known than the $n=0,1$ cases; a good reference where it can be found spelled out in detail is Section 5.4 of Lurie's Kerodon \cite{lurie-kerodon}.}
In these terms $2\mathrm{Vect}_\K$ is the 2-category whose 0-simplices are algebras $A_i$, whose 1-simplices are bimodules $M_{ij}$, where $M_{ij}$ is a left $A_i$-module and a right $A_j$-module, whose 2-simplices are morphisms of bimodules
\[
\varphi_{ijk}\colon M_{ij}\otimes_{A_j}M_{jk} \to M_{ik},
\]
whose 3-simplices are commutative tetrahedra, i.e., commutative diagrams of the form
\[
\begin{tikzcd}
M_{ij}\otimes_{A_j}M_{jk}\otimes_{A_k}M_{kl}\arrow[d,"\mathrm{id}\otimes \varphi_{jkl}"']\arrow[r,"\varphi_{ijk}\otimes\mathrm{id}"]&M_{ik}\otimes_{A_k}M_{kl}\arrow[d," \varphi_{ikl}"]\\
M_{ij}\otimes_{A_j}M_{jl}\arrow[r,"\varphi_{ijl}"]&M_{il}
\end{tikzcd}.
\]
Higher simplices are {defined by the condition that the nerve of a 2-category is 3-coskeletal}, i.e., for $k\geq 4$ a $k$-simplex is in the nerve of $2\mathrm{Vect}_\K$ if and only if all of its faces are in the nerve.
\par
By $\K^\ast$ we denote the multiplicative group of the field $\K$, and by 
$\mathbf{B}^2\K^\ast$ the 2-group given by the double delooping of the abelian group $\K^\ast$, i.e., the 2-groupoid having a single object with only the identity 1-morphism and with $\K^\ast$ as automorphism group of the identity 1-morphism.
There is a canonical embedding 
$
\iota\colon \mathbf{B}^2\K^\ast \hookrightarrow 2\mathrm{Vect}_\K
$
that maps the unique object of $\mathbf{B}^2\K^\ast$ to the $\K$-algebra $\K$, the identity 1-morphism in $\mathbf{B}^2\K^\ast$ to the $(\K,\K)$-bimodule $\K$, and the elements of $\K^\ast$ to themselves seen as $(\K,\K)$-bimodule morphisms from $\K$ to $\K$. In simplicial terms,  $\mathbf{B}^2\K^\ast$ is the simplicial abelian group whose 2-simplices are decorated by elements {$\alpha$} in $\K^\ast$  and whose 3-simplices are characterized by the 2-cocycle condition
{
$
\alpha_{023}\alpha_{012}=\alpha_{013}\alpha_{123}.
$
}
The realization of $\mathbf{B}^2\K^\ast$ as subsimplicial object of $2\mathrm{Vect}_\K$ is manifest.
\par
Through the whole paper we assume familiarity with basic construtions in higher category theory. A comprehensive treatment can be found in \cite{lurie-kerodon}.

\section{The standard $\mathbf{B}\K^\ast$-action on $\mathrm{Vect}_{\K}$}\label{sec:Vect-over-BK}
The 2-group $\mathbf{B}\K^\ast$ has a distinguished action on the symmetric monoidal category 
$\mathrm{Vect}_\K$, i.e., there is a symmetric monoidal $(2,1)$-category\footnote{I.e., a 2-category whose $k$-morphisms for $k>1$ are invertible.} $\mathrm{Vect}_\K/\!/\mathbf{B}\K^\ast$ fitting into a homotopy pullback diagram of the form
\begin{equation}\label{diagr:Bk-ast-action2}
            \begin{tikzcd}
            \mathrm{Vect}_\K
                \arrow[r]\arrow[d]\arrow[dr,phantom,"\lrcorner",very near start]&\ast\arrow[d]\\
            \mathrm{Vect}_\K/\!/\mathbf{B}\K^\ast\arrow[ur,Rightarrow,shorten <=2mm, shorten >=2mm]
                \arrow[r]&\mathbf{B}^2\K^\ast
            \end{tikzcd}.
\end{equation}
The symmetric monoidal category {$ \mathrm{Vect}_\K/\!/\mathbf{B}\K^\ast$}
will be the natural target for projective representations.\footnote{The fact that this target is naturally a 2-category justifies the motto that ``The theory of projective representations is a piece of 2-category theory fallen into the realm of 1-category theory.''.}
The construction begins by noticing that the 2-category $2\mathrm{Vect}_\K$ is naturally pointed, with pointing given by the unit object $\K$. One can then form the loop space 2-category $\Omega2\Vect_\K$, 
defined as the
lax\footnote{Here and in what follows we will always say `lax' to mean the 2-cell is not necessarily invertible; when the adjective `lax' is omitted it is meant that the 2-cell is invertible. To avoid possible confusion, we will occasionally write `strong' to stress that a certain homotopy commutativity is with an invertible 2-cell.}  homotopy pullback
\begin{equation}\label{eq:omega2vect}
            \begin{tikzcd}\Omega2 \Vect_\K
                \arrow[r]\arrow[d]\arrow[dr,phantom,"\lrcorner",very near start]&\ast\arrow[d]\\
                \ast\arrow[ur,Rightarrow,shorten <=2mm, shorten >=2mm]
                \arrow[r]&2\mathrm{Vect}_\K.
            \end{tikzcd}
        \end{equation}
    { One has 
    a natural} equivalence of symmetric monoidal categories $\Omega2\Vect_\K\cong \Vect_\K$. In particular, $\Omega2\mathrm{Vect}_\K$ which is a priori a 2-category is actually equivalent to a 1-category.

\begin{definition}
The symmetric monoidal $(2,1)$-category $\mathrm{Vect}_\K/\!/\mathbf{B}\K^\ast$ is defined by  the lax homotopy pullback
\begin{equation}\label{eq:defining-vect-mod-bk}
            \begin{tikzcd}\mathrm{Vect}_\K/\!/\mathbf{B}\K^\ast
                \arrow[r]\arrow[d]\arrow[dr,phantom,"\lrcorner",very near start]&\mathbf{B}^2\K^\ast\arrow[d,"\iota"]\\
                \ast\arrow[ur,Rightarrow,shorten <=2mm, shorten >=2mm]
                \arrow[r]&2\mathrm{Vect}_\K.
            \end{tikzcd}
        \end{equation}
\end{definition}
Here, the symmetric monoidal structure comes from the fact that the natural embedding $\iota\colon\mathbf{B}\K^\ast\hookrightarrow 2\mathrm{Vect}_\K$ is symmetric monoidal, and the fact that $\mathrm{Vect}_\K/\!/\mathbf{B}\K^\ast$ is a $(2,1)$-category from the fact that all $k$-morphisms in $\mathbf{B}\K^\ast$ for $k>1$ and all $k$-morphisms in $2\mathrm{Vect}_\K$ for $k>2$ are invertible. This will be manifest in the simplicial description of  $\mathrm{Vect}_\K/\!/\mathbf{B}\K^\ast$ we provide below.   
By the pasting law for lax homotopy pullbacks, from the defining diagram \eqref{eq:defining-vect-mod-bk}  we obtain the diagram
\begin{equation}\label{diagr:vect-over-bkast}
            \begin{tikzcd}
            \mathrm{Vect}_\K
                \arrow[r]\arrow[d]\arrow[dr,phantom,"\lrcorner",very near start]&\ast\arrow[d]\\
            \mathrm{Vect}_\K/\!/\mathbf{B}\K^\ast\arrow[ur,Rightarrow,shorten <=2mm, shorten >=2mm]
                \arrow[r]\arrow[d]\arrow[dr,phantom,"\lrcorner",very near start]&\mathbf{B}^2\K^\ast\arrow[d,"\iota"]\\
                \ast\arrow[ur,Rightarrow,shorten <=2mm, shorten >=2mm]\arrow[r]&2\mathrm{Vect}_\K,
            \end{tikzcd}
        \end{equation}
where the 2-out-of-3 rule has been used to identify the top left corner with the category $\mathrm{Vect}_\K$ of vector spaces over $\K$.
In the top square\footnote{Notice that this top square is automatically a strong homotopy pullback, since $\mathbf{B}^2\K^\ast$ is a 2-groupoid.} of \eqref{diagr:vect-over-bkast} we find the announced diagram \eqref{diagr:Bk-ast-action2}, encoding a higher group action of the 2-group $\mathbf{B}\K^\ast$ on the linear category $\mathrm{Vect}_\K$. We will refer to this action as to the  
 standard 2-action of $\mathbf{B}\K^\ast$ on $\mathrm{Vect}_\K$.
 { One easily obtains} the following explicit description of the $(2,1)$-category $\mathrm{Vect}_\K/\!/\mathbf{B}\K^\ast$.
\begin{lemma}\label{lem:explicit-proj}
 The nerve of the $(2,1)$-category $\mathrm{Vect}_\K/\!/\mathbf{B}\K^\ast$ is given by the following data:
 \begin{itemize}
 \item 0-simplices are (finite dimensional) vector spaces over $\K$;
 \item 1-simplices are morphisms $\phi_{ij}\colon V_i\to V_j$ of vector spaces;
 \item 2-simplices are homotopy commutative diagrams of the form
\[
            \begin{tikzcd}
            V_i
                \arrow[rr, "\phi_{ik}"]\arrow[dr,"\phi_{ij}"']&{}&
                V_k\\
                {}&V_j\arrow[u,shorten <=3pt,shorten > =1pt,Rightarrow,"\alpha_{ijk}"',pos=.5]\ar[ur,"\phi_{jk}"']&{}\\
            \end{tikzcd},
        \]
i.e., are the datum of three morphisms of $\K$-vector spaces $\phi_{ij}\colon V_i\to V_j$ and of a scalar $\alpha_{ijk}\in \K^\ast$ such that
\[
 \phi_{ik}\cdot \alpha_{ijk} =\phi_{jk}\circ\phi_{ij};
\]
\item 3-simplices are characterized by the fact that the decorations $\alpha_{ijk}$ of the faces satisfy the 2-cocycle condition
\[
\alpha_{ikl}\alpha_{ijk}=\alpha_{ijl}\alpha_{jkl}
\]
\item higher simplices are {determined by the 3-coskeletality condition}, i.e., for $k\geq 4$ one has
\[
\Delta^k(\mathrm{Vect}_\K/\!/\mathbf{B}\K^\ast)=(\partial \Delta^k)(\mathrm{Vect}_\K/\!/\mathbf{B}\K^\ast).
\]

 \end{itemize}
\end{lemma}
\begin{remark}
In terms of nerves, the morphism $\mathrm{Vect}_\K/\!/\mathbf{B}\K^\ast\to\mathbf{B}^2\K^\ast$ in diagram \eqref{diagr:Bk-ast-action2} simply consists in forgetting all 0-simplex and 1-simplex decorations, and retaining the 2-simplex decorations.
\end{remark}

\begin{remark}\label{rem:no-reference}
The content of Lemma \ref{lem:explicit-proj} can be taken as a definition of the $(2,1)$-category     $\mathrm{Vect}_\K/\!/\mathbf{B}\K^\ast$. This allows defining $\mathrm{Vect}_\K/\!/\mathbf{B}\K^\ast$ with no reference to $2\Vect_\K$. 
\end{remark}

\begin{remark}\label{rem:monoidal-structure}
The symmetric monoidal structure on  $\mathrm{Vect}_\K$ { naturally} induces a symmetric monoidal structure on $\mathrm{Vect}_\K/\!/\mathbf{B}\K^\ast$.
\end{remark}

\section{Projective representations of categories}\label{sec:proj-rep-categories}
The 2-category $\Vect_\K/\!/\mathbf{B}\K^\ast$ is the natural target of projective representations: given a category $\mathcal{C}$.\footnote{Here and in what follows by category we will usually mean $\infty$-category, and by functor we will mean $\infty$-functor, using 1-category and 1-functor when we want to stress that we are in the non higher context. Clearly, every 1-category is also an $\infty$-category and so a category in our use of the term, so expressions like ``the category of vector spaces'' will be non-ambiguous.}

\begin{definition}
    Let $\mathcal{C}$ be a category. A $\K$-linear representation of $\mathcal{C}$ is a functor $\rho\colon \mathcal{C}\to \Vect_\K$. A \emph{projective representation} of $\mathcal{C}$ (over $\K$) is a functor $\rho\colon \mathcal{C}\to \Vect_\K/\!/\mathbf{B}\K^\ast$.
A 2-cocycle on $\mathcal{C}$ with values in $\K^\ast$ is a functor $\alpha\colon \mathcal{C}\to \mathbf{B}^2\K^\ast$.    
\end{definition}
Every projective representation has an associated 2-cocycle, given by the composition of $\rho\colon \mathcal{C}\to \Vect_\K/\!/\mathbf{B}\K^\ast$ with the canonical morphism $\mathrm{Vect}_\K/\!/\mathbf{B}\K^\ast\to \mathbf{B}^2\K^\ast$. More generally, we can give the following.
\begin{definition}\label{def:proj-class-alpha}
Let $\alpha \colon \mathcal{C}\to \mathbf{B}^2\K^\ast$ be a $\K^\ast$-valued 2-cocycle on $\mathcal{C}$. A \emph{projective representation of $\mathcal{C}$ of class $\alpha$} is a lax homotopy commutative diagram of the form
\begin{equation*}
            \begin{tikzcd}
           \mathcal{C}
                \arrow[r, "\rho"]\arrow[dr,"\alpha"']&\mathrm{Vect}_\K/\!/\mathbf{B}\K^\ast\arrow[d]\arrow[ld,shorten > =45pt,Rightarrow,"\beta",pos=.1]\\
            {}&\mathbf{B}^2\K^\ast
            \end{tikzcd}.
        \end{equation*}
\end{definition}
\begin{remark}
When the filler $\beta$ is the identity, this precisely says that $\rho$ is a projective representation of $\mathcal{C}$ with associated $2$-cocycle $\alpha$. For a general $\beta$, the projective representation $
\rho$ will have a $2$-cocycle equivalent to $\alpha$ via the equivalence $\beta$. 
\end{remark}

\begin{remark}
Since $\Vect_\K/\!/\mathbf{B}\K^\ast$ is a 2-category, a projective representation of $\mathcal{C}$ will factor through the 2-truncation of $\mathcal{C}$. Hence will not be restrictive to assume $\mathcal{C}$ is a 2-category from the beginning, and we often will tacitly make this assumption, e.g., in not discussing the data for $k$-simplices of $\mathcal{C}$ for $k\geq 4$.
\end{remark}
\begin{remark}
Spelling out the definition, the functor $\rho$ of a projective representation of $\mathcal{C}$ associates with an object $X_i$ of $\mathcal{C}$ an object $V_{X_i}$ of $\mathrm{Vect}_\K/\!/\mathbf{B}\K^\ast$, i.e., a $\K$-vector space; with a 1-morphism $f_{ij}\colon X_i\to X_j$ in $\mathcal{C}$ a linear map
\[
\rho_{f_{ij}}\colon V_{X_i}\to V_{X_j}
\]
and with a 2 simplex 
\[
            \begin{tikzcd}
            X_i
                \arrow[rr, "f_{ik}"]\arrow[dr,"f_{ij}"']&{}&
                X_k\\
                {}&X_j\arrow[u,shorten <=3pt,shorten > =1pt,Rightarrow,"\Xi_{ijk}"',pos=.5]\ar[ur,"f_{jk}"']&{}\\
            \end{tikzcd},
\]
of $\mathcal{C}$ an element $\alpha_{\Xi_{ijk}}\in \K^\ast$ such that
$
\rho_{f_{jk}}\circ \rho_{f_{ij}} = \rho_{f_{ik}}\cdot \alpha_{\Xi_{ijk}}.
$
Functoriality of $\rho$ also tells us that for a a 3-simplex in $\mathcal{C}$ we have
$
\alpha_{\Xi_{ikl}}\alpha_{\Xi_{ijk}}=\alpha_{\Xi_{ijl}}\alpha_{\Xi_{jkl}},
$
that is the equation expressing the fact that $\alpha$ is a 2-cocycle on $\mathcal{C}$ with values in $\K^\ast$. 
\end{remark}

\begin{example}\label{ex:BG}
When $\mathcal{C}=\mathbf{B}G$, the above definition precisely reproduces the classical definition of a projective representation of the group $G$, and the definition of 2-cocycle reproduces that of a group 2-cocycle with values in $\K^\ast$ as a trivial $G$-module. The equivalence class of the 2-cocycle $\alpha\colon \mathbf{B}G\to \mathbf{B}^2\K^\ast$ is an element in the second group cohomology group $H^2_{\mathrm{Grp}}(G,\K^\ast)$ of $G$, and saying that a projective representation $\rho$ is of class $\alpha$ precisely means that the 2-cocycle associated with the projective representation $\rho$ is in the same cohomology class as $\alpha$. More precisely, these two 2-cocycles are related by the coboundary corresponding to the filler $\beta$ of Definition \ref{def:proj-class-alpha}. 
\end{example}

\begin{definition}
Let $\alpha\colon \mathcal{C}\to \mathbf{B}^2\K^\ast$ be a 2-cocycle. A \emph{trivialization} of $\alpha$ is a homotopy commutative diagram of the form
\[
   \begin{tikzcd}
            {}&\ast\ar[dr]&{}\\
            \mathcal{C}
                \arrow[rr, "\alpha"']\arrow[ur]&{}\arrow[u,shorten <=3pt,shorten > =3pt,Rightarrow,"\mathrm{\beta}"',pos=.5]&
                \mathbf{B}^2\K^\ast
            \end{tikzcd}.
\]
\end{definition}
\begin{example}
When $\mathcal{C}=\mathbf{B}G$, a trivialization in the above sense is precisely the datum of a $\K^\ast$-valued 1-cochain on $G$ with $d_{\mathrm{Grp}}\beta=\alpha$, where $d_{\mathrm{Grp}}$ is the differential of the usual cochain complex computing group cohomology.    
\end{example}

\begin{lemma}
 Let $\rho\colon \mathcal{C}\to \Vect_\K/\!/\mathbf{B}\K^\ast$ be a projective representation, and let $\alpha$ be its associated 2-cocycle. Then a trivialization of $\alpha$ is equivalent to a lift of $\rho$ to a linear representation $\hat{\rho}\colon \mathcal{C}\to \Vect_\K$.  
\end{lemma}
\begin{proof}
    Since $\alpha$ is the 2-cocycle associated with $\rho$, a trivialization of $\alpha$ is a homotopy commutative diagram of the form
    \begin{equation}\label{eq:to-be-factored}
    \begin{tikzcd}
            \mathcal{C}
                \arrow[r]\arrow[d,"\rho"']&\ast\arrow[d]\\
            \mathrm{Vect}_\K/\!/\mathbf{B}\K^\ast\arrow[ur,Rightarrow,shorten <=2mm, shorten >=2mm, "\beta"']
                \arrow[r]&\mathbf{B}^2\K^\ast
            \end{tikzcd}.
    \end{equation}
    By the universal property of lax homotopy pullback and by the top square in \eqref{diagr:vect-over-bkast}, diagram \eqref{eq:to-be-factored} uniquely\footnote{As always in the context of higher categories, uniqueness is up to homotopies, that are unique up to 2-homotopies, that are unique up to 3-homotopies, et cetera.} factors as
\begin{equation}\label{eq:need-outer}
            \begin{tikzcd}
            \mathcal{C}\ar[dr,"\hat{\rho}"]\arrow[drr,bend left=20]\arrow[ddr,bend right=20,"\rho"']&{}&{}\\
            {}&\mathrm{Vect}_\K\arrow[ru,shorten > =37pt,Rightarrow]
                \arrow[r]\arrow[d]\arrow[dr,phantom,"\lrcorner",very near start]
                &\ast\arrow[d]\\
                {}\arrow[ru,shorten < =43pt,Rightarrow,"\hat{\beta}",,pos=.87]&\mathrm{Vect}_\K/\!/\mathbf{B}\K^\ast\arrow[ur,Rightarrow,shorten <=2mm, shorten >=2mm]\arrow[r]&\mathbf{B}^2\K^\ast
            \end{tikzcd},
        \end{equation}   
where the filler in the bottom right square has no name since it is the canonical filler for the lax homotopy pullback and the filler in the top triangle has no name since $\ast$ is the terminal 2-category and so there is a unique\footnote{Again, uniqueness is up to homotopy: one means that the space of all such triangles is contractible.} triangle as the top one. Vice versa, given a lift $\hat{\rho}$ of $\rho$ to a linear representation, i.e., a lax homotopy commutative diagram of the form     \[
            \begin{tikzcd}
            \mathcal{C}\arrow[r,"\hat{\rho}"]\arrow[dr,"\rho"']& \mathrm{Vect}_\K\arrow[d]\\
            {}\arrow[ru,shorten < =43pt,Rightarrow,"\hat{\beta}",,pos=.87]&\mathrm{Vect}_\K/\!/\mathbf{B}\K^\ast
            \end{tikzcd},
        \]  
we can uniquely form a diagram of the form \eqref{eq:need-outer} (again by the uniqueness of the upper right triangle in \eqref{eq:need-outer}), and the outer diagram of this is a diagram of the form  \eqref{eq:to-be-factored}, i.e., a trivialization of $\alpha$.       
\end{proof}
\begin{remark}
Spelling out the definition, we see that a trivialization of $\alpha$ consists in associating  with any morphism $f_{ij}\colon X_i\to X_j$ in $\mathcal{C}$ an element $\beta_{f_{ij}}$ in $\mathbb{K}^*$ in such a way that for any  2-simplex 
\[
            \begin{tikzcd}
            X_i
                \arrow[rr, "f_{ik}"]\arrow[dr,"f_{ij}"']&{}&
                X_k\\
                {}&X_j\arrow[u,shorten <=3pt,shorten > =1pt,Rightarrow,"\Xi_{ijk}"',pos=.5]\ar[ur,"f_{jk}"']&{}\\
            \end{tikzcd},
\]
in $\mathcal{C}$  the equations
 $\beta_{f_{ik}}\alpha_{\Xi_{ijk}} =\beta_{jk}\beta_{ij}$
are satisfied.
 The lift $\hat{\rho}$ associated with the trivialization $\beta$ is 
    $
    \hat{\rho}_{f_{ij}}=\beta_{f_{ij}}^{-1}\rho_{f_{ij}}.
    $
    One then easily directly checks that $\hat{\rho}$ is a linear representation of $\mathcal{C}$.
    \end{remark}
\begin{remark}\label{rem:tensor-prod-proj-rep}
The monoidal structure on  $\mathrm{Vect}_\K/\!/\mathbf{B}\K^\ast$ from Remark \ref{rem:monoidal-structure} induces a tensor product on projective representations. One immediately sees that if $\rho_1$ and $\rho_2$ are projective representations of $\mathcal{C}$ of classes $\alpha_1$ and $\alpha_2$, respectively, then $\rho_1\otimes \rho_2$ is a projective representation of $\mathcal{C}$ of class $\alpha_1 \alpha_2$. It follows that if two projective representations $\rho$ and $\eta$ of $\mathcal{C}$ are such that their associated 2-cocycles are inverse to each other then $\rho\otimes \eta$ lifts to a linear representation of $\mathcal{C}$. 
\end{remark}
\begin{remark}\label{rem:proj-1-dim}
Let $\rho$ be a projective representation of $\mathcal{C}$ with associated 2-cocycle $\alpha$. Assume $\alpha$ has a trivialization $\beta$. Then from $\beta$ we can construct a projective representation $\eta_\beta$ as follows
\begin{align*}
\eta_\beta\colon X_i &\mapsto \K\\
f_{ij}&\mapsto \beta_{f_{ij}}^{-1}\\
\Xi_{ijk}&\mapsto \beta_{f_{ij}}^{-1}\beta_{f_{jk}}^{-1}\beta_{f_{ik}}.
\end{align*}
We can then form the tensor product $\eta_\beta\otimes \rho$.
Since $\beta$ is a trivialization of $\alpha$, 
the 2-cocycle for the projective representation  $\eta_\beta$ is $\alpha_{\Xi_{ijk}}^{-1}$. So, by Remark \ref{rem:tensor-prod-proj-rep}, $\eta_\beta\otimes \rho$ is a linear representation of $\mathcal{C}$. This is no surprise: making $\eta_\beta\otimes \rho$ explicit we find
\begin{align*}
\eta_\beta\otimes \rho \colon X_i &\mapsto \K\otimes V_{X_i}=V_{X_i}\\
f_{ij}&\mapsto \beta_{f_{ij}}^{-1}\rho_{f_{ij}}=\hat{\rho}_{f_{ij}}\\
\Xi_{ijk}&\mapsto \beta_{f_{ij}}^{-1}\beta_{f_{jk}}^{-1}\beta_{f_{ik}}\alpha_{Xi_{ijk}}=1.
\end{align*}
That is, $\eta_\beta\otimes \rho$ is precisely the linear lift $\hat{\rho}$ of $\rho$ associated with the trivialization $\beta$. This shows that linear lifts of projective representations can be seen as particular instances of tensor products of projective representations. 
\end{remark}

\section{Anomalous representations of categories}\label{sec:anomalousC}
A closely related notion to that of a projective representation of $\mathcal{C}$ with $2$-cocycle $\alpha$ is that of \emph{anomalous representation} with an anomaly $J$ endowed with an $\alpha$-structure $\xi$. 
\begin{definition}
Let $J\colon \mathcal{C}\to 2\Vect_\K$ be a functor. An \emph{anomalous representation} of $\mathcal{C}$ with anomaly $J$  is a lax 
homotopy commutative diagram of the form
\begin{equation}\label{diagram-bottom2}
            \begin{tikzcd}
            \mathcal{C}
                \arrow[rr, "J"]\arrow[dr]&{}&
                2\mathrm{Vect}_\K\\
                {}&\ast\arrow[u,shorten <=3pt,shorten > =1pt,Rightarrow,"Z"',pos=.5]\ar[ur]&{}\\
            \end{tikzcd}.
        \end{equation}
\end{definition}
\begin{remark}\label{rem:Z-explicit}
 It is useful to make fully explicit the definition of an anomalous representation $Z$ of $\mathcal{C}$ with anomaly $J$. { This explicitation will serve as a model for making fully explicit all of the constructions presented by universal properties in the following of the present article. On several occasions we will directly use these explicit descriptions in the proofs of a few results. The interested reader can easily derive these explicit descriptions from the one presently given here, or find them in \cite{vuppulury}.  
  An anomalous representation $Z$ of $\mathcal{C}$ with anomaly $J$} consists of
 \begin{itemize}
   \item a \emph{left} $J(X_i)$-module $Z(X_i)$, thought of as a $(J(X_i),\K)$-bimodule, for every object (0-simplex) $X_i$ of $\mathcal{C}$;  
   \item a filler $Z(f_{ij})$ for the diagram
\[
            \begin{tikzcd}
            \K
                \arrow[r,"\K"]\arrow[d,"Z(X_i)"']&\K\arrow[d,"Z(X_j)"]\\
           J(X_i)\arrow[ur,Rightarrow,shorten <=2mm, shorten >=2mm,"Z(f_{ij})"]
                \arrow[r,"J(f_{ij})"']& J(X_j)
            \end{tikzcd}
\]
in $2\Vect_\K$, i.e., a morphism of left $J(X_j)$-modules 
\[
Z(f_{ij})\colon J(f_{ij})\otimes_{J(X_i)} Z(X_i)\to Z(X_j),
\]
for any morphism $f_{ij}\colon X_i\to X_j$ in $\mathcal{C}$;
\end{itemize}
such that the prism
\[\begin{tikzcd}
	\K &&& {\text{ }} & \K \\
	& {\text{ }} & \K && {\text{ }} \\
	{\text{ }} & {\text{ }} && {\text{ }} \\
	J(X_i) && {\text{ }} && J(X_k) \\
	&& J(X_j) & {\text{ }} \\
	& {\text{ }}
	\arrow["\K", from=1-1, to=1-5]
	\arrow[""{name=0, anchor=center, inner sep=0}, "\K"', from=1-1, to=2-3]
	\arrow["Z(X_i)"', from=1-1, to=4-1]
	\arrow["Z(X_k)", from=1-5, to=4-5]
	\arrow["\K"', from=2-3, to=1-5]
	\arrow["Z(X_j)"',from=2-3, to=5-3]
	\arrow["Z(f_{ij})"', shift left=4, shorten <=17pt, shorten >=11pt, Rightarrow, from=4-1, to=2-3]
	\arrow[no head, from=4-1, to=4-3]
	\arrow[""{name=1, anchor=center, inner sep=0}, "J(f_{ij})"', from=4-1, to=5-3]
	\arrow["Z(f_{jk})", shift left=3, shorten <=17pt, shorten >=11pt, Rightarrow, from=4-3, to=2-5]
	\arrow[""{name=2, anchor=center, inner sep=0}, "J(f_{ik})", from=4-3, to=4-5]
	\arrow["J(f_{jk})"', from=5-3, to=4-5]
	\arrow["1"', shift right, shorten <=17pt, shorten >=11pt, Rightarrow, from=0, to=1-4]
	\arrow["J(\Xi_{ijk})"{pos=0.3}, shorten <=6pt, shorten >=47pt, Rightarrow, from=1, to=2]
\end{tikzcd}\]
commutes in $2\Vect_\K$, i.e., such that the diagram of morphisms of left $J(X_k)$-modules
\begin{equation}\label{eq:diagram:Z-J}
\begin{tikzcd}[column sep=large]
            J(f_{jk})\otimes_{J(X_j)}J(f_{ij})\otimes Z(X_i) \arrow[d,"J(\Xi_{ijk})\otimes{\mathrm{id}}"']\arrow[r,"\mathrm{id}\otimes Z(f_{ij})"]&J(f_{jk})\otimes_{J(X_j)}Z(X_j)
            \arrow[d,"Z(f_{jk})"]\\
            J(f_{ik})\otimes_{J(X_i)} Z(X_i)\arrow[r,"Z(f_{ik})"]&Z(X_k)
            \end{tikzcd}
\end{equation}
commutes, for any  2-simplex 
\[
 \begin{tikzcd}
            & {X_j} \arrow[dr,"{f_{jk}}"]&\\
            {X_i}\arrow[ru,"{f_{ij}}"]\arrow[rr,"{f_{ik}}"']&\arrow[u,Leftarrow,shorten <=2mm, shorten >=2mm,"\Xi_{ijk}"']{}&{X_k}
            \end{tikzcd}.
\]
of $\mathcal{C}$
\end{remark}

We now introduce the notion of invertible anomalies and of $\alpha$-structures on a given anomaly, where $\alpha$ is a 2-cocycle. The interplay between these two notions will lead to Proposition \ref{prop:proj-vs-anomalous}, relating anomalous and projective representations.

\begin{definition}
The Picard 3-group $\Pic(2\Vect_\K)\subseteq 2\Vect_\K$ is the 3-group of invertible algebras, invertible bimodules and invertible morphism of bimodules inside the Morita 2-category. We say that an anomaly $J\colon \mathcal{C}\to 2\Vect_\K$ is \emph{invertible} if it factors through $\Pic(2\Vect_\K)\subseteq 2\Vect_\K$, i.e., if one is given the datum of a strong\footnote{i.e., the 2-cell components are all invertible. Another terminology to express this consists in saying that the 2-cell filler is a ``pseudonatural transformation'', as opposed to ``lax natural transformation''.} homotopy commutative diagram of the form
\begin{equation}\label{diagram-top2a}
            \begin{tikzcd}
            {}&\Pic(2\Vect_K)\ar[dr,"\iota"]&{}\\
            \mathcal{C}
                \arrow[rr, "J"']\arrow[ur]&{}\arrow[u,shorten <=3pt,shorten > =3pt,Rightarrow,pos=.5]&
                2\mathrm{Vect}_\K
            \end{tikzcd},
        \end{equation}
 where $\iota\colon \Pic(2\Vect_\K)\to 2\Vect_K$ is the inclusion.       
\end{definition}
\begin{definition}
Let $J\colon \mathcal{C}\to 2\Vect_\K$ be a functor, and let $\alpha\colon \mathcal{C}\to \mathbf{B}^2\K^\ast$ be a 2-cocycle on $\mathcal{C}$ with values in $\K^\ast$. An \emph{$\alpha$-structure} on $J$ is the datum of a {\color{black}lax} homotopy commutative diagram of the form
\begin{equation}\label{diagram-top2}
            \begin{tikzcd}
            {}&\mathbf{B}^2\K^\ast\ar[dr,"\iota"]&{}\\
            \mathcal{C}
                \arrow[rr, "J"']\arrow[ur,"\alpha"]&{}\arrow[u,shorten <=3pt,shorten > =3pt,Rightarrow,"\mathrm{\xi}"',pos=.5]&
                2\mathrm{Vect}_\K
            \end{tikzcd}.
        \end{equation}
We say that an $\alpha$ structure is invertible if  \eqref{diagram-top2} is strong.    
\end{definition}

\begin{remark}\label{rem:inclusions}
By composing a 2-cocycle $\alpha\colon \mathcal{C}\to \mathbf{B}^2\K^\ast$ with the inclusion $\iota\colon \mathbf{B}^2\K^\ast\hookrightarrow 2\Vect_\K$ one gets a functor
\[
J_\alpha=\iota\circ\alpha\colon \mathcal{C}\to 2\Vect_\K
\]
The anomaly $J_\alpha$ is endowed with a canonical invertible $\alpha$-structure: the one with identity filler $\xi$. We will call $J_\alpha$, together with its canonical $\alpha$-structure, the canonical anomaly associated with the 2-cocycle $\alpha$. Since the inclusion $\mathbf{B}^2\K^\ast\hookrightarrow 2\Vect_\K$ factors as
  \[
  \mathbf{B}^2\K^\ast\hookrightarrow\mathbf{B}\Pic(\Vect_\K)\hookrightarrow \Pic(2\Vect_\K)\hookrightarrow 2\Vect_\K.
  \]
 we see that an anomaly $J$ endowed with an invertible $\alpha$-structure is automatically invertible. In particular, $J_\alpha$ is an invertible anomaly.  
\end{remark}

\begin{proposition}\label{prop:proj-vs-anomalous}
Let $\alpha\colon \mathcal{C}\to \mathbf{B}^2\K^\ast$ be a 2-cocycle on $\mathcal{C}$ with values in $\K^\ast$. 
Then we have an equivalence {between the 
category of projective representations of $\mathcal{C}$ of class $\alpha$ and the category
of anomalous representations of $\mathcal{C}$ with anomaly $J_\alpha$.
More generally, if $\xi$ is an \emph{invertible} $\alpha$-structure on $J\colon \mathcal{C}\to 2\Vect_\K$, then $\xi$  induces an equivalence between the 
category of  projective representations of $\mathcal{C}$ of class $\alpha$ and the category of anomalous representations of $\mathcal{C}$ with anomaly $J$.}
\end{proposition}        
\begin{proof}
By pasting \eqref{diagram-top2} with \eqref{diagram-bottom2}, we get the lax homotopy commutative diagram
\begin{equation}\label{diagram:induces-proj2}
            \begin{tikzcd}
            \mathcal{C}
                \arrow[r,"\alpha"]\arrow[d]&\mathbf{B}^2\K^\ast\arrow[d,"\iota"]\\
            \ast\arrow[ur,Rightarrow,shorten <=2mm, shorten >=2mm,"\xi\circ Z"']
                \arrow[r]&2\mathrm{Vect}_\K
            \end{tikzcd}.
        \end{equation}

        By the universal property of lax pullbacks, this diagram can uniquely be factored as
\begin{equation*}
            \begin{tikzcd}
            \mathcal{C}\ar[dr,"\rho^{}_{\xi\circ Z}"]\arrow[drr,bend left=20,"\alpha"]\arrow[ddr,bend right=20]&{}&{}\\
            {}&\mathrm{Vect}_\K/\!/\mathbf{B}\K^\ast\arrow[ru,shorten > =37pt,Rightarrow,"\beta_{\xi\circ Z}",pos=.2]
                \arrow[r]\arrow[d]\arrow[dr,phantom,"\lrcorner",very near start]
                &\mathbf{B}^2\K^\ast\arrow[d,"\iota"]\\
                {}\arrow[ru,shorten < =43pt,Rightarrow]&\ast\arrow[ur,Rightarrow,shorten <=2mm, shorten >=2mm]\arrow[r]&2\mathrm{Vect}_\K
            \end{tikzcd}.
        \end{equation*}        
From the top triangle in the above diagram, comparing with Definition \ref{def:proj-class-alpha}, one sees that 
\[
\rho_{\xi\circ Z}\colon \mathcal{C}\to \mathrm{Vect}_\K/\!/\mathbf{B}\K^\ast
\]
is a projective representation of $\mathcal{C}$ of class $\alpha$. Vice versa, since $\ast$ is the terminal $2$-category, given a projective representation of $\mathcal{C}$ of class $\alpha$ we can uniquely form the diagram 
\begin{equation*}
            \begin{tikzcd}
            \mathcal{C}\ar[dr,"\rho"]\arrow[drr,bend left=20,"\alpha"]\arrow[ddr,bend right=20]&{}&{}\\
            {}&\mathrm{Vect}_\K/\!/\mathbf{B}\K^\ast\arrow[ru,shorten > =37pt,Rightarrow,"\beta",pos=.2]
                \arrow[r]\arrow[d]\arrow[dr,phantom,"\lrcorner",very near start]
                &\mathbf{B}^2\K^\ast\arrow[d,"\iota"]\\
                {}\arrow[ru,shorten < =43pt,Rightarrow]&\ast\arrow[ur,Rightarrow,shorten <=2mm, shorten >=2mm]\arrow[r]&2\mathrm{Vect}_\K
            \end{tikzcd}.
        \end{equation*}
whose outer square is a lax homotopy commutative diagram of the form
\begin{equation}\label{eq:for-use-now}
            \begin{tikzcd}
            \mathcal{C}
                \arrow[r,"\alpha"]\arrow[d]&\mathbf{B}^2\K^\ast\arrow[d,"\iota"]\\
            \ast\arrow[ur,Rightarrow,shorten <=2mm, shorten >=2mm,"Z_{\rho,\beta}",pos=.6]
                \arrow[r]&2\mathrm{Vect}_\K
            \end{tikzcd},
        \end{equation}
Since the 2-cell $\xi$ is invertible, we can write         
 $Z_{\rho,\beta}=\xi \circ (\xi^{-1}\circ Z_{\rho,\beta})$, and \eqref{eq:for-use-now} is factored as  
 \begin{equation*}
            \begin{tikzcd}[column sep=large]
            {}&\mathbf{B}^2\K^\ast\ar[dr,"\iota"]&{}\\
            \mathcal{C}\arrow[ru,"\alpha"]
                \arrow[rr, "J"]\arrow[dr]&{}\arrow[u,shorten <=10pt,shorten > =3pt,Rightarrow,"\xi"',pos=.6]&
                2\mathrm{Vect}_\K\\
                {}&\ast \arrow[u,shorten <=3pt,shorten > =2pt,Rightarrow,"\, \xi^{-1}\circ Z_{\rho,\beta}"',pos=.8]
                \ar[ur]&{}\\
            \end{tikzcd},
        \end{equation*}
thus exhibiting $\xi^{-1}\circ Z_{\rho,\beta}$ as an anomalous representation of $\mathcal{C}$ with anomaly $J$. By uniqueness, the two constructions are inverse to each other, so that when the $\alpha$-structure $\xi$ on $J$ is invertible we find an equivalence. 
\end{proof}

\begin{remark}\label{rem:no-invertibility-needed}
The first part of the proof of Proposition \ref{prop:proj-vs-anomalous} only uses the datum of the $\alpha$-structure on $J$ and not its invertibility. So we still have { that
anomalous representations of $\mathcal{C}$ with anomaly $J$ induce
projective representations of $\mathcal{C}$ of class $\alpha$}
under the sole assumption $\xi$ is an $\alpha$-structure on $J$. However, this will not be an equivalence unless $\xi$ is invertible.
\end{remark}
\begin{remark}\label{rem:rho-Z-explicit}
{ For an object $X_i$ in $\mathcal{C}$, we have
 $
 \rho_{\xi\circ Z}(X_i)=\xi(X_i)\otimes_{J(X_i)}Z(X_i)
 $.
  For a 1-morphism $f_{ij}\colon X_i\to X_j$ in $\mathcal{C}$ we have that $ \rho_{\xi\circ Z}(f_{ij})$ is the 
 linear map 
\[
\xi(X_i)\otimes_{J(X_i)}Z(X_i)\xrightarrow{\xi(f_{ij})\otimes \mathrm{id}_{Z(X_i)} } \xi(X_j)\otimes_{J(X_j)}J(f_{ij})\otimes_{J(X_i)}Z(X_i)\xrightarrow{\mathrm{id}_{\xi(X_j)}\otimes Z(f_{ij})}\xi(X_j)\otimes_{J(X_j)}Z(X_j).
\]
}
\end{remark}
\begin{remark}
It is reasonable to expect that the constructions of the previous sections make sense with an arbitrary 2-term $\Omega$ sequence of symmetric monoidal $\infty$-categories, i.e., for any sequence $((n+1)\mathcal{V},n\mathcal{V},(n-1)\mathcal{V})$ where
\begin{itemize}
    \item $k\mathcal{V}$ is a symmetric monoidal $(\infty,k)$-category;
    \item $(k-1)\mathcal{V}\cong \Omega (k\mathcal{V})$.
\end{itemize}
Writing simply $\mathcal{V}$ for $1\mathcal{V}$, the sequence we have been considering in the previous sections has been 
\begin{align*}
0\mathcal{V}&=0\Vect_\K:\text{ the commutative algebra $\K$ of ``numbers'', that can be thought of as }\\
&\phantom{=0\Vect_\K:\quad}\text{``0-vector spaces'' }\\
\mathcal{V}&=\Vect_\K:\text{the symmetric monoidal category of $\K$-vector spaces}\\
2\mathcal{V}&=2\Vect_\K\text{the symmetric monoidal Morita 2-category of ``2-vector spaces''}
\end{align*}
A second immediate example is the super-version of this $\Omega$-triple:
\begin{align*}
0\mathcal{V}&=0\mathrm{sVect}_\K: \text{ the commutative algebra $\K$ of ``numbers'', that can be thought of as}\\
&\phantom{=0\mathrm{sVect}_\K\quad}\text{ ``super (i.e., $\mathbb{Z}/2\mathbb{Z}$-graded) 0-vector spaces'' }\\
\mathcal{V}&=\mathrm{sVect}_\K: \text{the symmetric monoidal category $\mathrm{sVect}_\K$ of super  $\K$-vector spaces}\\
2\mathcal{V}&=2\mathrm{sVect}_\K: \text{the symmetric monoidal Morita 2-category of ``super 2-vector spaces''.}
\end{align*}
Here, more explicitly, $2\mathrm{sVect}_\K$ is the  2-category whose objects are (finite dimensional) super (i.e., $\mathbb{Z}/2\mathbb{Z}$-graded) $\K$-algebras, whose 1-morphism are (finite dimensional) $\mathbb{Z}/2\mathbb{Z}$-graded bimodules and whose 2-morphisms are morphisms of $\mathbb{Z}/2\mathbb{Z}$-graded bimodules. As in the non-super case, composition of 1-morphisms is given by tensor products, which in this case is the tensor product of $\mathbb{Z}/2\mathbb{Z}$-graded bimodules.
\par
Other examples are given by
the $\Omega$-triple
\begin{align*}
0\mathcal{V}&=0\mathrm{Hilb}_\K: \text{ the commutative algebra $\mathbb{C}$ of complex numbers, thought of as}\\
&\phantom{=0\mathrm{Hilb}_\K: \quad }\text{``0-Hilbert spaces'' }\\
\mathcal{V}&=\mathrm{Hilb}_\K: \text{the symmetric monoidal category $\mathrm{Hilb}$ of Hilbert spaces}\\
2\mathcal{V}&=2\mathrm{Hilb}: \text{the symmetric monoidal Morita 2-category of `Hilbert 2-vector spaces'',}
\end{align*}
modelled as the 2-category of von Neumann algebras, Hilbert bimodules and continuous intertwiners, as well by its super version, see \cite{schmidpeter} for details.

\par
In the general setting, the chain of inclusions appearing in  Remark \ref{rem:inclusions} becomes
\[
\mathbf{B}^2\Pic((n-1)\mathcal{V})\hookrightarrow \mathbf{B}\Pic(n\mathcal{V})\hookrightarrow \Pic((n+1)\mathcal{V})\hookrightarrow (n+1)\mathcal{V}
\]
and all of the constructions of the previous sections should verbatim translate to this more general setting. In particular the target with projective $n\mathcal{V}$-valued representations would be the $(\infty,n-1)$-category 
\[
n\mathcal{V}/\!/\mathbf{B}\Pic((n-1) \mathcal{V}).
\]
 A somehow more explicit notation for the same $(\infty,n-1)$-category is 
\[
n\mathcal{V}/\!/\mathbf{B}\mathrm{Aut}_{n\mathcal{V}}(\mathbf{1}_{n\mathcal{V}}),
\]
where one stresses the fact that $\Pic((n-1)\mathcal{V})$ is the group of automorphisms of the unit object of the $(\infty,n)$-symmetric monoidal category $n\mathcal{V}$. Also notice that the possibility of multiplying a morphism of vector spaces with an invertible scalar, that played a prominent role into giving explicit equations for projective representations with values in $\Vect_\K$ makes perfect sense for any $(\infty,n)$- symmetric monoidal category $n\mathcal{V}$: the multiplication 
\[
\K^\ast\times \mathrm{Hom}_{\Vect_\K}(V_i,V_j)\to \mathrm{Hom}_{\Vect_\K}(V_i,V_j)
\]
is a particular instance of the multiplication
\[
\mathrm{Aut}_{n\mathcal{V}}(\mathbf{1}_{n\mathcal{V}})\times \mathrm{Hom}_{n\mathcal{V}}(V_i,V_j)\to \mathrm{Hom}_{\color{black}{n\mathcal{V}}}(V_i,V_j)
\]
given by
\begin{align*}
\mathrm{Aut}_{n\mathcal{V}}(\mathbf{1}_{n\mathcal{V}})\times \mathrm{Hom}_{n\mathcal{V}}(V_i,V_j)&\xrightarrow{\otimes} 
\mathrm{Hom}_{n\mathcal{V}}(\mathbf{1}_{n\mathcal{V}}\otimes V_i,\mathbf{1}_{n\mathcal{V}} \otimes V_j) \cong \mathrm{Hom}_{n\mathcal{V}}(V_i,V_j)\\
(\alpha,f)&\mapsto \alpha\otimes f
\end{align*}
Also notice that invertibility actually plays no role here: one could have $\K$ and $\mathrm{End}_{n\mathcal{V}}(\mathbf{1}_{\mathcal{V}})$ in place of $\K^\ast$ and  $\mathrm{Aut}_{n\mathcal{V}}(\mathbf{1}_{n\mathcal{V}})$ in the above formulas.

One could therefore have given all definitions and constructions from the beginning in the general setup considered in this Remark. Yet, a close inspection of the relevant literature showed that many of the expected results these definitions and constructions would be based on in the generality of symmetric monoidal $(\infty,n)$-categories are{, for $n\geq 3$,} more part of a well-established folklore than being available in the form of rigorously established results. So we preferred to limit ourselves to the treatment of the concrete example of $\K$-vector spaces, so to have the constructions presented in the paper based on solid ground. The construction immediately extend to the other explicit examples mentioned above, i.e., super-2-vector spaces and  (super-)2-Hilbert spaces; the rigorous extension to arbitrary symmetric monoidal $(\infty,n)$-categories is the subject of \cite{vuppulury2}.
\par
Focusing on $2\Vect_\K$ also has the benefit not to lose the reader into abstractness from the very beginning, and to have a more immediate recognition  of known results and construction from the classical theory of projective representation of groups. At the same time, we tried to present all constructions in a sufficient abstract way, using only categorical properties of $\Vect_\K$, to make the transition to the general setting of of symmetric monoidal $(\infty,n)$-categories immediate once the needed basic results on $\infty$-group actions on  symmetric monoidal $(\infty,n)$-categories are rigorously established. 
\end{remark}

\section{From anomalous representations to linear representations}

Let $G$ be a finite group, and let $\alpha$ be a $\K^\ast$-valued 2-cocycle on $G$. Then one can use $\alpha$ to construct a $\K^\ast$-central extension $G^\alpha$ of $G$ by the rule: elements of $G^\alpha$ are pairs $(g,\lambda)$ with $g\in G$ and $\lambda\in \K^\ast$ and with the multiplication given by
\[
(g_1,\lambda_1)\cdot (g_2,\lambda_2)=(g_1g_2,\alpha(g_1,g_2)\lambda_1\lambda_2).
\]
In categorical terms, the construction of $G^\alpha$ is a homotopy pullback: the groupoid $\mathbf{B}G^\alpha$ is the homotopy pullback
\begin{equation}\label{diagr:def-G-alpha}
            \begin{tikzcd}
            {\mathbf{B}G^\alpha}
                \arrow[r]\arrow[d]\arrow[dr,phantom,"\lrcorner",very near start]&\mathbf{B}G\arrow[d,"\alpha"]\arrow[dl,Rightarrow,shorten <=2mm, shorten >=2mm]\\
           \ast
                \arrow[r]&\mathbf{B}^2\K^\ast
            \end{tikzcd}.
 \end{equation}
 \begin{remark}
 Notice that we are saying ``homotopy pullback'' here instead of ``lax homotopy pullback'' since the category in the bottom right corner is a 2-groupoid, and so the filler of the 2-cell in \eqref{diagr:def-G-alpha} is automatically invertible. 
 \end{remark}
A well know result from the theory of projective representations of group is that projective representations of $G$ with 2-cocycle $\alpha$ are equivalent to linear representations of the central extension $G^\alpha$. Actually, stated this way, the result is improperly stated. To see this just consider the case where $G=\{e\}$ is the trivial group. Then also $\alpha$ is trivial and $G^\alpha=\K^\ast$. The statement would then say that projective representations of the trivial group, that are necessarily trivial, are equivalent to linear representations of the group $\K^\ast$, that are not necessarily trivial (already the 1-dimensional representation of $\K^\ast$ acting on $\K$ as scalars is nontrivial). The correct statement is: projective representations of $G$ with 2-cocycle $\alpha$ are equivalent to linear representations of the central extension $G^\alpha$ such that the subgroup $\K^\ast\subseteq G^\alpha$ acts as scalars. In this and in the following Section we are going to present a generalization of this result to the case where $\mathbf{B}G$ is replaced by an arbitrary category and $\alpha$ by an arbitrary anomaly. As in the previous Sections, we will present the constructions with concrete linear targets such as $\Vect_\K$ and $2\Vect_\K$, but in such a way that a generalization to an arbitrary symmetric monoidal 2-category $\mathcal{V}$
replacing $2\Vect_\K$ should be immediate. There will be however also a few specificities of $2\Vect_\K$, in particular the fact that every algebra is a module over itself, that will play a role. 

\subsection{From anomalous representations to linear representations}\label{sec:anomalous-to-linearC}

We begin by defining the main character of this Section, {namely the category $\mathcal{C}^J$}, by generalizing diagram \eqref{diagr:def-G-alpha}. {This amounts to performing a Grothendieck construction, so that the properties of  $\mathcal{C}^J$ are an instance of the straightening/unstraightening yoga in higher category theory, see \cite{htt}.
}
\begin{definition}
Let $\mathcal{C}$ be a category and $J\colon\mathcal{C}\to 2\Vect_\K$ be a functor. The \emph{extension of $\mathcal{C}$ with anomaly $J$} is the category $\mathcal{C}^J$ defined by the   lax homotopy pullback
\begin{equation}\label{diagr:def-CJ}
            \begin{tikzcd}
            {\mathcal{C}^{J}}
                \arrow[r]\arrow[d]\arrow[dr,phantom,"\lrcorner",very near start]&\mathcal{C}\arrow[d,"J"]\arrow[dl,Rightarrow,shorten <=2mm, shorten >=2mm]\\
           \ast
                \arrow[r]&2\mathrm{Vect}_\K
            \end{tikzcd}.
        \end{equation}
\end{definition}  
\begin{remark}\label{rem:unwind-cj}
{Objects of $\mathcal{C}^J$ are pairs $(X_i,L_{X_i})$, where $X$ is an object in $\mathcal{C}$ and  $L_{X_i}$ is a \emph{right} $J(X_i)$-module, and morphisms in $\mathcal{C}^J$ are pairs $(f_{ij},\varphi_{f_{ij}})$, where $f_{ij}\colon X_i\to X_j$ is a morphism in $\mathcal{C}$ and $\varphi_{f_{ij}}$ is 
a morphism of right $J(X_i)$-modules $\varphi_{f_{ij}}\colon L_{X_i} \to L_{X_j}\otimes_{J(X_j)}J(f_{ij})$. 
Finally, with any 2-morphism  $\Xi_{ijk}$ in $\mathcal{C}$ it is associated a 
commutative diagram of morphisms of right $J(X_i)$-modules
\begin{equation}\label{eq:diagram-L-J}
\begin{tikzcd}[column sep=large]
            L_{X_i} \arrow[d,"{\varphi_{f_{ij}}}"']\arrow[r,"\varphi_{f_{ik}}"]&L_{X_k}\otimes_{J(X_k)}J(f_{ik})
            \\
            L_{X_j}\otimes_{J(X_j)}J(f_{ij})\arrow[r,"{\varphi_{f_{jk}}\otimes\mathrm{id}}"]&L_{X_k}\otimes_{J(X_k)}J(f_{jk})\otimes_{J(X_j)}J(f_{ij})\arrow[u,"\mathrm{id}\otimes J(\Xi_{ijk})"']
            \end{tikzcd}.
\end{equation}
}
\end{remark}

\begin{remark}
If $\mathcal{C}$ is a symmetric monoidal category and $J$ is a symmetric monoidal functor, then $\mathcal{C}^J$ is naturally endowed with a symmetric monoidal structure as well.    
\end{remark}
\begin{remark}\label{rem:towards-EW}
Let $X$ be an object of $\mathcal{C}$. Then we have an embedding
\[
\iota_X\colon \mathrm{Mod}_{J(X)}\to \mathcal{C}^J
\]
given by $L\mapsto (X,L)$ at the level of objects. At the level of 1-morphism, $\iota_X$ is
\[
\iota_X(L_i\xrightarrow{\varphi_{ij}}L_j) = (X,L_i)\xrightarrow{(\mathrm{id}_X,\varphi_{ij})}(X,L_j).
\]
Here one is using the fact that $J(\mathrm{id}_X)=J(X)$ as a $(J(X),J(X))$-bimodule, so that $\varphi_{ij}$ is a morphism from $L_i$ to $L_j\otimes_{J(X)}J(\mathrm{id}_X)$. Finally, at the level of 2-simplices the functor $\iota_X$ acts as
\[
\iota_X\left(
\begin{tikzcd}
            & {L_j} \arrow[dr,"{\varphi_{jk}}"]&\\
            {L_i}\arrow[ru,"{\varphi_{ij}}"]\arrow[rr,"{\varphi_{ik}}"']&{}&{L_k}
            \end{tikzcd}\right)=
 \begin{tikzcd}
            & {(X,L_{j})} \arrow[dr,"{(\mathrm{id}_X,\varphi_{{jk}})}"]&\\
            {(X_i,L_{X_i})}\arrow[ru,"{(\mathrm{id}_X,\varphi_{{ij}})}"]\arrow[rr,"{(\mathrm{id}_X,\varphi_{{ik}})}"']&\arrow[u,Leftarrow,shorten <=2mm, shorten >=2mm,"\mathrm{id}"']{}&{(X,L_{k})}
            \end{tikzcd}.
\]
If  $E\colon \mathcal{C}^J\to \Vect_\K$ is a functor, then we can restrict it along $\iota_X$, i.e., consider  the functor
$E\bigr\vert_X=E\circ \iota_X\colon \mathrm{Mod}_{J(X)}\to \Vect_\K$.
This way, once $E$ is fixed, every object $X$ of $\mathcal{C}$ defines a functor from right $J(X)$-modules to $\K$-vector spaces.
\end{remark}

{\begin{definition}\label{def:towards-EW}
Let $E\colon \mathcal{C}^J\to \Vect_\K$ be a functor. We say that $E$ is \emph{additive and cocontinuous over} $\mathcal{C}$ if all of the functors  $E\bigr\vert_X\colon \mathrm{Mod}_{J(X)}\to \Vect_\K$ from Remark \ref{rem:towards-EW} are additive and cocontinuous (i.e., preserving small colimits). 
\end{definition}
}

\begin{remark}
 The functor $\iota_X$ from Remark \ref{rem:towards-EW} identifies the lax fiber of $\mathcal{C}^J\to \mathcal{C}$ over the object $X$ with a copy of the category $\mathrm{Mod}_{J(X)}$ of right $J(X)$-modules. This identification can be canonically seen from the pasting law for lax homotopy pullbacks: {\color{black} in the diagram}
\begin{equation*}
            \begin{tikzcd}
                {\mathrm{Mod}_{J(X)}}\arrow[r]\arrow[d]\arrow[dr,phantom,"\lrcorner",very near start]& \mathcal{C}^{J}\arrow[r]\arrow[d]\arrow[dr,phantom,"\lrcorner",very near start]&\ast\arrow[d]\\
                \ast\arrow[ur,Rightarrow,shorten <=2mm, shorten >=2mm]\arrow[r,"X"']& \mathcal{C}\arrow[ur,Rightarrow,shorten <=2mm, shorten >=2mm]\arrow[r,"{J}"']&2\mathrm{Vect}_\K.
            \end{tikzcd}
        \end{equation*}
        both the rightmost square (by definition of $\mathcal{C}^{J}$) and the outer square are  lax homotopy pullbacks; by the 2-out-of-3 property also the leftmost square is a  lax homotopy pullback.
\end{remark}

\begin{remark}\label{rem:tensored-over-vect}
The category $\mathcal{C}^J$ is naturally tensored over $\Vect_\K$: the $\Vect_\K$-action on $\mathcal{C}^J$ is given by
$(V,(X_i,L_{X_i}))\mapsto (X_i, V\otimes_\K L_{X_i})$.
\end{remark}

By Remark \ref{rem:tensored-over-vect}, both $\mathcal{C}^J$ and $\mathrm{Vect}_\K$ are $\mathrm{Vect}_\K$-linear categories, so one can look at $\mathrm{Vect}_\K$-linear functors $E$ from $\mathcal{C}^J$ to $\mathrm{Vect}_\K$. Putting this together with Definition \ref{def:towards-EW}, we give the following definition, clearly hinting at the Eilenberg--Watts theorem \cite{eilenberg,watts}.
\begin{definition}\label{def:Vect-linear-cocontinuous}
Let $\mathcal{C}$ be a category and $J\colon\mathcal{C}\to 2\Vect_\K$ be a functor. 
We write $\mathrm{Hom}_{\mathrm{Vect}_\K}(\mathcal{C}^{J},\mathrm{Vect}_\K)$ for the category of $\mathrm{Vect}_\K$-linear functors that are additive and cocontinuous  over $\mathcal{C}$. 
\end{definition}
Clearly $\mathrm{Hom}_{\mathrm{Vect}_\K}(\mathcal{C}^{J},\mathrm{Vect}_\K)\subseteq \mathrm{Hom}(\mathcal{C}^{J}, \mathrm{Vect}_\K)$, that is, every element of \allowbreak $\mathrm{Hom}_{\mathrm{Vect}_\K}(\mathcal{C}^{J}, \allowbreak\mathrm{Vect}_\K)$ is in particular a linear representation of $\mathcal{C}^J$. The $\Vect_\K$-linearity makes the elements of $\mathrm{Hom}_{\mathrm{Vect}_\K}(\mathcal{C}^{J},\allowbreak \mathrm{Vect}_\K)$ special among the linear representations of $\mathcal{C}^J$; as we are going to see in detail in the subsequent section, this condition is akin to the ``$\K^\ast$ acting as scalars'' condition for the linear representations of $G^\alpha$ corresponding to projective representations of $G$ we mentioned at the beginning of this Section.
\begin{proposition}\label{prop:anom-to-lin}
 There is a distinguished natural map {$Z\mapsto E_Z$, from
 anomalous representations of $\mathcal{C}$ with anomaly $J$ to $\mathrm{Hom}_{\mathrm{Vect}_\K}(\mathcal{C}^{J},\mathrm{Vect}_\K)$,}
defined as follows in terms of the notation of Remark \ref{rem:Z-explicit}. 
\begin{itemize}
\item
For an object $(X_i,L_{X_i})$ of $\mathcal{C}^J$ one has
\begin{equation}\label{eq:EZ2-obj}
E_Z(X_i,L_{X_i})= L_{X_i}\otimes_{J(X_i)} Z(X_i),
\end{equation}
\item For a 1-morphism $(f_{ij},\varphi_{f_{ij}})$ in $\mathcal{C}^J$, the morphism $E_Z(f_{ij},\varphi_{f_{ij}})$ in $\Vect_\K$ is the composition
\begin{equation}\label{eq:EZ2-morph}
L_{X_i}\otimes_{J(X_i)} Z(X_i) \xrightarrow{\varphi_{f_{ij}}\otimes \mathrm{id}_{Z(X_i)}}
L_{X_j}\otimes_{J(X_j)}J(f_{ij})\otimes_{J(X_i)} Z(X_i)\xrightarrow{\mathrm{id}_{L_{X_j}}\otimes Z(f_{ij})} L_{X_j}\otimes_{J(X_j)}Z(X_j)
\end{equation}
\end{itemize}
\end{proposition}
\begin{proof}
By pasting diagram \eqref{diagram-bottom2} to diagram \eqref{diagr:def-CJ} we get a lax homotopy commutative diagram of the form
\begin{equation}\label{diagr:Z-hat2}
            \begin{tikzcd}
            \mathcal{C}^{J}
                \arrow[r]\arrow[d]&\ast\arrow[d]\arrow[dl,Rightarrow,shorten <=2mm, shorten >=2mm,"\hat{Z}"]\\
            \ast
                \arrow[r]&2\mathrm{Vect}_\K
            \end{tikzcd}.
        \end{equation}
By the universal property of lax homotopy pullbacks, 
this uniquely\footnote{As usual and already remarked, uniqueness is in a homotopical sense: it means that the space of these factorizations is contractible.} factors as     
\begin{equation*}
            \begin{tikzcd}
            \mathcal{C}^{J}\ar[dr,"E_Z"]\arrow[drr,bend left=20]\arrow[ddr,bend right=20]&{}&{}\arrow[ld,shorten < =37pt,Rightarrow]\\
            {}&\mathrm{Vect}_\K
                \arrow[r]\arrow[d]\arrow[dr,phantom,"\lrcorner",very near start]
                \arrow[ld,shorten > =37pt,Rightarrow]&\ast\arrow[d]\arrow[dl,Rightarrow,shorten <=2mm, shorten >=2mm]\\
                {}&\ast\arrow[r]&2\mathrm{Vect}_\K
            \end{tikzcd},
        \end{equation*}
giving a morphism $
E_Z\colon \mathcal{C}^{J}\to \mathrm{Vect}_\K$.
We conclude the proof by showing that on objects and 1-morphisms in $\mathcal{C}^J$ the functor $E_Z$ has the announced behavior. From this, the $\mathrm{Vect}_\K$-linearity of $E_Z$ is manifest.
Let $(X_i,L_{X_i})$ be an object in $\mathcal{C}^{J}$. As emphasized in Remark \ref{rem:Z-explicit}, the 2-cell $Z$ from diagram \eqref{diagram-bottom2} associates with $X_i$ the left $J(X_i)$-module $Z(X_i)$, seen as a $(J(X_i),\K)$-bimodule. Since diagram \eqref{diagr:Z-hat2} is the pasting of diagram \eqref{diagram-bottom2} with diagram \eqref{diagr:def-CJ}, the 2-cell $\hat{Z}$ from diagram \eqref{diagr:Z-hat2} associates with $(X_i,L_{X_i})$ the composition of 1-morphisms
$
\K\xrightarrow{Z(X_i)}J(X_i)\xrightarrow{L_{X_i}}\K
$
in $2\Vect_\K$, i.e., the $\K$-vector space
$
L_{X_{i}}\otimes_{J(X_i)}Z(X_i),
$
seen as a $(\K,\K)$-bimodule. The description of $E_Z$ on 1-morphisms in $\mathcal{C}^J$ goes along the very same lines. Finally, for every object $X_i$ of $\mathcal{C}$, we have
$
E_Z\bigr\vert_{X_i}=-\otimes_{J(X_i)}Z(X_i),
$
so $E_Z\bigr\vert_{X_i}$ is both additive and cocontinuous.
\end{proof}

\begin{remark}
In the proof of Proposition \ref{prop:anom-to-lin} we didn't explicitly check that $E_Z$ maps 2-simplices in $\mathcal{C}^J$ to commutative diagrams of $\K$-vector spaces. Since $E_Z$ is constructed from universal properties, this is automatically satisfied. Yet, it is a simple check to verify that the association $E_Z$ explicitly defined by equations \eqref{eq:EZ2-obj}--\eqref{eq:EZ2-morph} does indeed map 2-simplices in $\mathcal{C}^J$ to commutative diagrams of $\K$-vector spaces, thus defining a functor $\mathcal{C}^J\to \Vect_\K$. That is, a reader not wishing to dwell into higher categorical structures may define $E_Z$ directly by \eqref{eq:EZ2-obj}--\eqref{eq:EZ2-morph}, without relying on the notion and properties of lax homotopy pullbacks. To see that \eqref{eq:EZ2-obj}--\eqref{eq:EZ2-morph} indeed define a functor, let 
\[
 \begin{tikzcd}
            & {(X_j,L_{X_j})} \arrow[dr,"{(f_{jk},\varphi_{f_{jk}})}"]&\\
            {(X_i,L_{X_i})}\arrow[ru,"{(f_{ij},\varphi_{f_{ij}})}"]\arrow[rr,"{(f_{ik},\varphi_{f_{ik}})}"']&\arrow[u,Leftarrow,shorten <=2mm, shorten >=2mm,"\Xi_{ijk}"']{}&{(X_k,L_{X_k})}
            \end{tikzcd}
\]
be a 2-simplex in $\mathcal{C}^J$. By \eqref{eq:EZ2-obj}--\eqref{eq:EZ2-morph}, the diagram of $\K$-vector spaces
\[
\begin{tikzcd}
            & {E_Z(X_j,L_{X_j})} \arrow[dr,"{E_Z(f_{jk},\varphi_{f_{jk}})}"]&\\
            {E_Z(X_i,L_{X_i})}\arrow[ru,"{E_Z(f_{ij},\varphi_{f_{ij}})}"]\arrow[rr,"{{\color{black}{E_Z}}(f_{ik},\varphi_{f_{ik}})}"']&{}&{E_Z(X_k,L_{X_k})}
            \end{tikzcd}
\]
is the diagram
\[
\begin{tikzcd}[column sep=-3.5em]
            & L_{X_{j}}\otimes_{J(X_j)}Z(X_j) \arrow[dr,"\varphi_{f_{jk}}\otimes \mathrm{id}_{Z(X_j)}"]&\\
            L_{X_j}\otimes_{J(X_j)}J(f_{ij})\otimes_{J(X_i)} Z(X_i)\arrow[ru,"\mathrm{id}_{L_{X_j}}\otimes Z(f_{ij})"]\arrow[dr,"\varphi_{f_{jk}}\otimes\mathrm{id}_{J(f_{ij})}\otimes \mathrm{id}_{Z(X_i)}"']&{}&L_{X_k}\otimes_{J(X_k)}J(f_{jk})\otimes_{J(X_j)} Z(X_j)\arrow[dd,"\mathrm{id}_{L_{X_k}}\otimes Z(f_{jk})"]\\
            & L_{X_{k}}\otimes_{J(X_k)}J(f_{jk})\otimes_{J(X_j)}J(f_{ij})\otimes_{J(X_i)}Z(X_i)\arrow[dd,"\mathrm{id}_{L_{X_k}}\otimes J(\Xi_{ijk})\otimes \mathrm{id}_{Z(X_i)}"] \arrow[ur,"\mathrm{id}_{L_{X_k}}\otimes \mathrm{id}_{J(f_{jk})}\otimes Z(f_{jk})"]&\\
            L_{X_{i}}\otimes_{J(X_i)}Z(X_i)\arrow[uu,"\varphi_{f_{ij}}\otimes \mathrm{id}_{Z(X_i)}"]\arrow[rd,"\varphi_{f_{ik}}\otimes \mathrm{id}_{Z(X_i)}"']&{}&L_{X_{k}}\otimes_{J(X_k)}Z(X_k)\\
            &L_{X_k}\otimes_{J(X_k)}J(f_{ik})\otimes_{J(X_i)} Z(X_i)\arrow[ur,"\mathrm{id}_{L_{X_k}}\otimes Z(f_{ik})"']
            \end{tikzcd}
\]
and this commutes: the commutativity of the lower left square is diagram \eqref{eq:diagram-L-J}, the commutativity of the lower right square is diagram \eqref{eq:diagram:Z-J}, and the commutativity of the top square is the functoriality of tensor product.
\end{remark}

\subsection{From linear representations to projective representations}\label{sec:linear-to projective2}
In the previous Section we have shown how an anomalous representation of $\mathcal{C}$ with anomaly $J$ induces a linear representation of $\mathcal{C}^J$. We now show how, given an $\alpha$-structure on $J$, a linear representation of 
$\mathcal{C}^J$ induces a projective representation of $\mathcal{C}$ of class $\alpha$. Putting these two constructions together we get a projective representation of $\mathcal{C}$ starting with an anomalous one endowed with an $\alpha$-structure. We have already described in Proposition \ref{prop:proj-vs-anomalous} another way of producing projective representations out of anomalous ones. No surprise, as we are going to show, these two constructions of projective representations out of anomalous ones are equivalent. 
This ultimately relies on the fact that, since the groupoid $\mathbf{B}\K^\ast$ is naturally a submonoidal category of $\Vect_\K$, a $\Vect_\K$-linear functor out of $\mathcal{C}^J$ is naturally $\mathbf{B}\K^\ast$-equivariant.

Along the same line {we used in Section \ref{sec:Vect-over-BK} we produced the 2-category $\Vect_\K/\!/\mathbf{B}^2\K^\ast$} we can define the 2-category ${\mathcal{C}^{J}/\!/\mathbf{B}\K^\ast}$ by factoring the defining diagram \eqref{diagr:def-CJ} of $\mathcal{C}^J$
as
\begin{equation}\label{diagr:def-CJoverKast}
            \begin{tikzcd}
            {\mathcal{C}^{J}}
                \arrow[r]\arrow[d]\arrow[dr,phantom,"\lrcorner",very near start]&\ast\arrow[d]\\
                {\mathcal{C}^{J}/\!/\mathbf{B}\K^\ast}\arrow[ur,Rightarrow,shorten <=2mm, shorten >=2mm]
                \arrow[r]\arrow[d]\arrow[dr,phantom,"\lrcorner",very near start]&\mathbf{B}^2\K^\ast\arrow[d,"\iota"]\\
            \mathcal{C}\arrow[ur,Rightarrow,shorten <=2mm, shorten >=2mm]
                \arrow[r,"J"']&2\mathrm{Vect}_\K
            \end{tikzcd}.
\end{equation}

\begin{lemma}\label{lemma:equivariant}
    Let $E\colon \mathcal{C}^J\to \Vect_\K$ be a $\Vect_\K$-linear functor. Then $E$ is $\mathbf{B}\K^\ast$-equivariant, i.e., it is part of a lax homotopy commutative diagram of the form
    \begin{equation}\label{diagr:withEZ2}
            \begin{tikzcd}
            {\mathcal{C}^{J}}
                \arrow[rr,"E"]\arrow[d,"\pi"']&{}&{\mathrm{Vect}_\K}\arrow[d]\\
                {\mathcal{C}^{J}/\!/\mathbf{B}\K^\ast}\arrow[urr,Rightarrow,shorten <=2mm, shorten >=2mm]
                \arrow[rr,"E/\!/\mathbf{B}\K^\ast"]\arrow[dr]&{}&\mathrm{Vect}_\K/\!/\mathbf{B}\K^\ast\arrow[dl]\\
            {}\arrow[urr,Rightarrow,shorten < =57pt]&\mathbf{B}^2\K^\ast&{}
            \end{tikzcd}.
            \end{equation}
\end{lemma}
\begin{proof}
We have to define the functor
$E/\!/\mathbf{B}\K^\ast\colon \mathcal{C}^J/\!/\mathbf{B}\K^\ast \to \Vect_\K/\!/\mathbf{B}\K^\ast$.
Since 0- and 1-simplices in $\mathcal{C}^J/\!/\mathbf{B}\K^\ast$ and in $\Vect_\K/\!/\mathbf{B}\K^\ast$ are the same as 0- and 1-simplices in $\mathcal{C}^J$ and in $\Vect_\K$, respectively,
we define $E/\!/\mathbf{B}\K^\ast\colon \mathcal{C}^J/\!/\mathbf{B}\K^\ast$ to coincide with $E$ on these simplices. For a 2-simplex
$((\Xi_{ijk},\alpha_{ijk}),\phi_{f_{ij}},\phi_{f_{jk}},\phi_{f_{ik}})$
in $\mathcal{C}^J/\!/\mathbf{B}\K^\ast$, we define its image via $E/\!/\mathbf{B}\K^\ast$
to be the 2-simplex
\[
\begin{tikzcd}
            & {E(X_j,L_{X_j})} \arrow[dr,"{E(f_{jk},\varphi_{f_{jk}})}"]\arrow[d,Rightarrow,shorten >=1mm, shorten <=1mm,"\alpha_{ijk}"]&\\
            {E(X_i,L_{X_i})}\arrow[ru,"{E(f_{ij},\varphi_{f_{ij}})}"]\arrow[rr,"{E(f_{ik},\varphi_{f_{ik}})}"']&{}&{E(X_k,L_{X_k})}
            \end{tikzcd}.
\]
{Using the $\Vect_\K$-linearity of $E$, one easily checks that this indeed is a 2-simplex in $\Vect_\K/\!/\mathbf{B}\K^\ast$.}
It is straightforward to see that $E/\!/\mathbf{B}\K^\ast$ maps 3-simplices of $\mathcal{C}^J/\!/\mathbf{B}\K^\ast$ to 3-simplices of $\Vect_\K$; { by 3-coskeletality this concludes the construction of $E/\!/\mathbf{B}\K^\ast$.}
The commutativity of \eqref{diagr:withEZ2} is manifest, with identity 2-cells. 
\end{proof}
\begin{corollary}
Let $Z\colon \mathcal{C}\to 2\Vect_\K$ be an anomalous representation with anomaly $J$. The morphism $E_Z\colon \mathcal{C}^J\to \Vect_K$ is $\mathbf{B}\K^\ast$-equivariant.    
\end{corollary}
\begin{lemma}
    Let $Z\colon \mathcal{C}\to 2\Vect_\K$ be an anomalous representation with anomaly $J$, and let $\xi$ be an $\alpha$-structure on $J$. Then $\xi$ induces a lax section $\hat{\xi}$ of the projection $\pi\colon \mathcal{C}^J\to \mathcal{C}^J/\!/\mathbf{B}\K^\ast$.
\end{lemma}
\begin{proof}
The defining diagram \eqref{diagram-top2} of the $\alpha$-structure $\xi$ can be read as a lax homotopy commutative diagram of the form
\begin{equation}\label{diagr:induces-hat2}
            \begin{tikzcd}
            {\mathcal{C}}
                \arrow[r,"\alpha"]\arrow[d,"\mathrm{id}"']&{\mathbf{B}^2\K^\ast}\arrow[d,"\iota"]\\
            \mathcal{C}\arrow[ur,Rightarrow,shorten <=2mm, shorten >=2mm,"\xi"']
                \arrow[r,"J"']&2\mathrm{Vect}_\K
            \end{tikzcd}.
        \end{equation}
By the universal property of the lax homotopy pullback diagram defining $\mathcal{C}^J/\!/\mathbf{B}\K^\ast$, (the bottom square in \eqref{diagr:def-CJoverKast}), diagram \eqref{diagr:induces-hat2} factors as
\begin{equation}\label{diagr:here-is-alpha}
            \begin{tikzcd}
            \mathcal{C}\ar[dr,"\hat{\xi}"]\arrow[drr,bend left=20,"\alpha"]\arrow[ddr,bend right=20,"\mathrm{id}"']&{}&{}\arrow[ld,shorten < = 34pt,Leftarrow]\\
            {}&\mathcal{C}^J/\!/\mathbf{B}\K^\ast
                \arrow[r]\arrow[d,"\pi"]\arrow[dr,phantom,"\lrcorner",very near start]
                \arrow[ld,shorten > =37pt,Leftarrow]&\mathbf{B}^2\K^\ast\arrow[d,"\iota"]\arrow[dl,Leftarrow,shorten <=2mm, shorten >=2mm]\\
                {}&\mathcal{C}\arrow[r]&2\mathrm{Vect}_\K
            \end{tikzcd}.
        \end{equation}
On the right of this diagram we read the lax homotopy commutative diagram        
\begin{equation*}
            \begin{tikzcd}
            {{\mathcal{C}^{J}/\!/\mathbf{B}\K^\ast}}
            \arrow[dd,"\pi", bend left=40]\\
            \\
            {\mathcal{C}}\arrow[uu,Rightarrow,shorten <=4mm,shorten >=4mm]\arrow[uu,"\hat{\xi}", bend left=40]
            \end{tikzcd}.
        \end{equation*} 
       telling us that $\hat{\xi}$ is a lax section of $\mathcal{C}^{J}/\!/\mathbf{B}\K^\ast\to \mathcal{C}$.
\end{proof}        
\begin{remark}\label{rem:explicit-section}
{For any morphism $f_{ij}\colon X_i\to X_j$ in $\mathcal{C}$ one has
$
\hat{\xi}(X_i)=(X_i,\xi(X_i))$ and 
$
\hat{\xi}(f_{ij})=(f_{ij},\xi(f_{ij}))$.
}
\end{remark}

\begin{corollary}\label{cor:proj-from-xi}
   Let $Z\colon \mathcal{C}\to 2\Vect_\K$ be an anomalous representation with anomaly $J$, and let $\xi$ be an $\alpha$-structure on $J$. Then with any $\mathrm{Vect}_\K$-linear morphism $E\colon \mathcal{C}^{J}\to \mathrm{Vect}_\K$ is naturally associated a projective representation of $\mathcal{C}$ of class $\alpha$.
\end{corollary}
\begin{proof}
With our data we can then form the composition
$
\mathcal{C}\xrightarrow{\hat{\xi}} {{\mathcal{C}^{J}/\!/\mathbf{B}\K^\ast}} \xrightarrow{E/\!/\mathbf{B}\K^\ast} \mathrm{Vect}_\K/\!/\mathbf{B}\K^\ast,
$
that gives a projective representation of $\mathcal{C}$. By \eqref{diagr:withEZ2}, the 2-cocycle associated with this projective representation is the composition 
$
\mathcal{C}\xrightarrow{\hat{\xi}} {{\mathcal{C}^{J}/\!/\mathbf{B}\K^\ast}} \to \mathbf{B}^2\K^\ast,
$
and by \eqref{diagr:here-is-alpha} this is $\alpha$.
\end{proof}

\subsection{Two ways of going from anomalous representations to projective ones}
Putting the pieces together we see we have actually exhibited two ways of going from anomalous representations endowed with $\alpha$-structures to projective representations of class $\alpha$. The first one has been described in Section \ref{sec:anomalousC} and consists in producing the functor
$
\rho^{}_{Z\circ \xi}\colon \mathcal{C}\to \mathrm{Vect}_\K/\!/\mathbf{B}\K^\ast
$
as described in Proposition \ref{prop:proj-vs-anomalous} and Remarks \ref{rem:no-invertibility-needed} and \ref{rem:rho-Z-explicit}. 
The second one consists in using the results from Section \ref{sec:linear-to projective2}: one first produces the $\Vect_\K$-linear representation $E_Z$ of $\mathcal{C}^J$ and then the projective representation
$
(E_Z/\!/\mathbf{B}\K^\ast)\circ \hat{\xi}\colon \mathcal{C}\to \mathrm{Vect}_\K/\!/\mathbf{B}\K^\ast
$
from Corollary \ref{cor:proj-from-xi}.
It will probably not be a surprise that these two projective representations are equivalent. We state and prove this in the following Proposition.
\begin{proposition}
    Let $Z\colon \mathcal{C}\to 2\Vect_\K$ be an anomalous representation with anomaly $J$, and let $\xi$ be an $\alpha$-structure on $J$. Then we have a natural equivalence of projective representations of class $\alpha$
    \begin{equation}\label{eq:two-is-one}
            \begin{tikzcd}
            & \mathcal{C}^J/\!/\mathbf{B}\K^\ast \arrow[dr,"E_Z/\!/\mathbf{B}\K^\ast"]&\\
            \mathcal{C}\arrow[ru,"\hat{\xi}"]\arrow[rr,"\rho_{Z\circ\xi}"']&\arrow[u,Leftarrow,shorten <=2mm, shorten >=2mm]{}&{\Vect_\K}/\!/\mathbf{B}\K^\ast
            \end{tikzcd}.
        \end{equation} 
\end{proposition}
\begin{proof}
We will recursively show that \eqref{eq:two-is-one} strongly homotopy commutes, with canonical 2-cell given by the identity.\footnote{So it strictly commutes.} By recursively, we mean that we will first show this on 0-simplices of $\mathcal{C}$, next on 1-simplices, and finally on 2-simplices. {By the 3-coskeletality of ${\Vect_\K}/\!/\mathbf{B}\K^\ast$}, this will conclude the proof.
We already have an explicit description of $\rho_{Z\circ\xi}$ from Remark \ref{rem:rho-Z-explicit}, so we just need to put together Proposition \ref{prop:anom-to-lin} with Remark \ref{rem:explicit-section} and compare.
 Let $X_i$ be an object of $\mathcal{C}$. Then  
 \[
(E_Z/\!/\mathbf{B}\K^\ast\circ \hat{\xi})(X_i)=E_Z(X_i,\xi(X_i))=
\xi(X_i)\otimes_{J(X_i)}Z(X_i),
\]
by \eqref{eq:EZ2-obj}. Therefore, $E_Z/\!/\mathbf{B}\K^\ast\circ \hat{\xi}$ coincides with $\rho_{Z\circ\xi}$ on objects.
\par
For a 1-morphism $f_{ij}\colon X_i\to X_j$ in $\mathcal{C}$,
\[
(E_Z/\!/\mathbf{B}\K^\ast\circ \hat{\xi})(f_{ij})=E_Z(f_{ij},\xi(f_{ij})),
\]
and so it is given by the composition
\[
\xi(X_i)\otimes_{J(X_i)} Z(X_i) \xrightarrow{\xi(f_{ij})\otimes \mathrm{id}_{Z(X_i)}}
\xi(X_j)\otimes_{J(X_j)}J(f_{ij})\otimes_{J(X_i)} Z(X_i)\xrightarrow{\mathrm{id}_{\xi(X_j)}\otimes Z(f_{ij})} \xi(X_j)\otimes_{J(X_j)}Z(X_j)
\]
by \eqref{eq:EZ2-morph}. Therefore, $E_Z/\!/\mathbf{B}\K^\ast\circ \hat{\xi}$ coincides with $\rho_{Z\circ\xi}$ on 1-morphisms. For a 2-simplex $(\Xi_{ijk},\phi_{f_{ij}},\phi_{f_{jk}},\phi_{f_{ik}})$
in $\mathcal{C}$, its image via $E_Z/\!/\mathbf{B}\K^\ast\circ \hat{\xi}$
is the 2-simplex
\[
\begin{tikzcd}
            & {E_Z(X_j,\xi(X_j))} \arrow[dr,"{E_Z(f_{jk},\xi(f_{jk}))}"]\arrow[d,Rightarrow,shorten >=1mm, shorten <=1mm,"\alpha_{ijk}"]&\\
            {E_Z(X_i,\xi(X_i))}\arrow[ru,"{E_Z(f_{ij},\xi(f_{ij}))}"]\arrow[rr,"{E_Z(f_{ik},\xi(f_{ik}))}"']&{}&{E_Z(X_k,\xi(X_k))}
            \end{tikzcd},
\]
and so the 2-simplex
\[
\begin{tikzcd}
            & {\xi(X_j)\otimes_{J(X_j)}Z(X_j)} \arrow[dr,"{E(f_{jk},\xi(f_{jk}))}"]\arrow[d,Rightarrow,shorten >=1mm, shorten <=1mm,"\alpha_{ijk}"]&\\
            {\xi(X_i)\otimes_{J(X_i)}Z(X_i)}\arrow[ru,"{E_Z(f_{ij},\xi(f_{ij}))}"]\arrow[rr,"{E_Z(f_{ik},\xi(f_{ik}))}"']&{}&{\xi(X_k)\otimes_{J(X_k)}Z(X_k)}
            \end{tikzcd},
\]
of $\Vect_\K/\!/\mathbf{B}\K^\ast$. Since in showing the coincidence of $E_Z/\!/\mathbf{B}\K^\ast\circ \hat{\xi}$ with $\rho_{Z\circ\xi}$ on 1-morphisms we have in particular shown that $E_Z(f_{ij},\xi(f_{ij}))=\rho_{Z\circ\xi}(f_{ij})$, this 2-simplex is
\[
\begin{tikzcd}
            & {\xi(X_j)\otimes_{J(X_j)}Z(X_j)} \arrow[dr,"{\rho_{Z\circ\xi}(f_{jk})}"]\arrow[d,Rightarrow,shorten >=1mm, shorten <=1mm,"\alpha_{ijk}"]&\\
            {\xi(X_i)\otimes_{J(X_i)}Z(X_i)}\arrow[ru,"{\rho_{Z\circ\xi}(f_{ij})}"]\arrow[rr,"{\rho_{Z\circ\xi}(f_{ik})}"']&{}&{\xi(X_k)\otimes_{J(X_k)}Z(X_k)}
            \end{tikzcd},
\]
showing that $E_Z/\!/\mathbf{B}\K^\ast\circ \hat{\xi}$ coincides with $\rho_{Z\circ\xi}$ on 2-simplices, too.
\end{proof}

\section{The Stolz--Teichner subcategory $\mathcal{C}^J_{\mathrm{ST}}$}

Let $G$ be a finite group, and let $\alpha$ be a $\K^\ast$-valued 2-cocycle on $G$, and $J_\alpha$ be the composition of $\alpha$ with the inclusion $\mathbf{B}^2\K^\ast\hookrightarrow 2\Vect_K$. Then we have seen in Proposition
\ref{prop:proj-vs-anomalous}
 that we have an equivalence   
{ between the category of projective representations of the group $\mathbf{B}G$ 
of class $\alpha$ and the category of anomalous representations of $\mathbf{B}G$ with anomaly $J_\alpha$, and 
in Proposition \ref{prop:anom-to-lin}
that there is a distinguished natural map from 
anomalous representations of $\mathbf{B}G$ with anomaly $J_\alpha$
to 
$\mathrm{Hom}_{\mathrm{Vect}_\K}((\mathbf{B}G)^{J_\alpha},\mathrm{Vect}_\K$.
} Putting these two together, and recalling that a projective representation of the groupoid $\mathbf{B}G$ in the sense of Section \ref{sec:proj-rep-categories} is precisely a projective representation of the group $G$ in the sense of group representations, we get 
a distinguished natural map
 { from projective representations of $G$ of class $\alpha$ to
$\mathrm{Hom}_{\mathrm{Vect}_\K}((\mathbf{B}G)^{J_\alpha},\mathrm{Vect}_\K)$.}
Comparing this with the classical { equivalence 
between projective representations of $G$ of class $\alpha$ and linear representations of $G^\alpha$ such that $\K^\ast$ acts as scalars}
suggests that the category $\mathbf{B}G^{J_\alpha}$ should be closely related to the central extension $G^\alpha$. \par
The naive guess $(\mathbf{B}G)^{J_\alpha}\cong \mathbf{B}G^\alpha$ is clearly wrong, since $(\mathbf{B}G)^{J_\alpha}$ is not even a groupoid. And even taking the core of $(\mathbf{B}G)^{J_\alpha}$, i.e., its maximal subgroupoid, we still would not have an equivalence: $(\mathbf{B}G)^{J_\alpha}$ has many more objects than $\mathbf{B}G^\alpha$ (that 
has a single object). Yet, as we are going to show in this section, the category $\mathcal{C}^J$ contains a distinguished subcategory $\mathcal{C}^J_{\mathrm{ST}}$, \emph{with the same objects as $\mathcal{C}$}, capturing all the information of $\mathrm{Vect}_\K$-linear functors out of $\mathcal{C}^J$. 
Specializing this to the case where $\mathcal{C}=\mathbf{B}G$ and $J=J_\alpha$, one finds a distinguished subcategory $(\mathbf{B}G)^{J_\alpha}_{\mathrm{ST}}$ of $\mathbf{B}G^{J_\alpha}$, that encodes all of the relevant information of $\mathbf{B}G^\alpha$.

More precisely, the fact that the 
subgroup $\K^\ast$ acts as scalars means that the action of $\K^\ast$ is induced by the inclusion $\K^\ast\subseteq \K$ and by the action of $\K$ as field of scalars for $\K$-vector spaces. This means that the linear representations of the group central extension $G^\alpha=\K^\ast\times_\alpha G$ such that $\K^\ast$ acts as scalars are precisely the linear representations of the \emph{monoid} central extension $M_\alpha=\K\times_\alpha G$ such that $\K$ acts as the field of scalars. Therefore we have an equivalence { between projective representations of $G$ of class $\alpha$
and linear representations of $M_\alpha$ such that $\K$ acts as scalars},
and what we are going to show is that there is an equivalence 
$
\mathbf{B}M_\alpha\cong (\mathbf{B}G)^{J_\alpha}_{\mathrm{ST}}$.

\begin{remark}The subscript ``ST'' stands for ``Stolz--Teichner''. This choice is dictated by the fact that, when $\mathcal{C}$ is the category of conformal spin bordisms and $J$ is the $n$-th tensor power of the so-called Fermionic (or Clifford/Fock) anomaly, the category $\mathcal{C}^J_{\mathrm{ST}}$ is the enriched bordism category considered by Stolz and Teichner in their description of Clifford linear field theories of degree $n$ in \cite{what-is-an-elliptic-object}.
\end{remark}
\begin{definition}\label{def:st-category}
    Let $J\colon \mathcal{C}\to 2\Vect_\K$ be a functor. The Stolz--Teichner category $\mathcal{C}^J_{\mathrm{ST}}$ is the category defined as follows.
    \begin{itemize}
        \item 0-simplices of $\mathcal{C}^J_{\mathrm{ST}}$ are the 0-simplices $X_i$ of $\mathcal{C}$.
        \item 1-simplices $X_i\xrightarrow{(f_{ij},v_{f_{ij}})} X_j$
        of $\mathcal{C}^J_{\mathrm{ST}}$ are pairs consisting of a 1-simplex $f_{ij}\colon X_i\to X_j$ of $\mathcal{C}$, together with a pointing, i.e., a distinguished element, $v_{f_{ij}}$ of the $(J(X_j),J(X_i))$-bimodule $J(f_{ij})$.
\item 2-simplices 
        \[
 \begin{tikzcd}
            & {X_j} \arrow[dr,"{(f_{jk},v_{f_{jk}})}"]&\\
            {X_i}\arrow[ru,"{(f_{ij},v_{f_{ij}})}"]\arrow[rr,"{(f_{ik},v_{f_{ik}})}"']&\arrow[u,Leftarrow,shorten <=2mm, shorten >=2mm,"\Xi_{ijk}"']{}&{X_k}
            \end{tikzcd}.
\]
of $\mathcal{C}^J_{\mathrm{ST}}$ are pairs consisting of a 2-simplex 
\[
 \begin{tikzcd}
            & {X_j} \arrow[dr,"{f_{jk}}"]&\\
            {X_i}\arrow[ru,"{f_{ij}}"]\arrow[rr,"{f_{ik}}"']&\arrow[u,Leftarrow,shorten <=2mm, shorten >=2mm,"\Xi_{ijk}"']{}&{X_k}
            \end{tikzcd}.
\]
of $\mathcal{C}$ and of a triple of 1-simplices
 \[
     X_i\xrightarrow{(f_{ij},v_{f_{ij}})}X_j, \qquad\qquad \qquad 
     \\
     X_j\xrightarrow{(f_{jk},v_{f_{jk}})}X_k,  \qquad\qquad \qquad 
     \\
     X_i\xrightarrow{(f_{ik},v_{f_{ik}})}X_k
 \]
 in $\mathcal{C}^J_{\mathrm{ST}}$ such that $
 J(\Xi_{ijk})(v_{f_{jk}}\otimes v_{f_{ij}})=v_{f_{ik}},
 $
 i.e., such that $
 J(\Xi_{ijk})\colon (J(f_{jk}),v_{f_{jk}})\otimes_{J(X_j)}(J(f_{ij}),v_{f_{ij}})\to (J(f_{ik}),v_{f_{ik}})$
 is a morphism of pointed bimodules;
 \item For $k\geq 3$, the set $\Delta^k(\mathcal{C}^J_{\mathrm{ST}})$ of $k$-simplices of $\mathcal{C}^J_{\mathrm{ST}}$ is defined recursively as the fiber product
\[
\Delta^k(\mathcal{C}^J_{\mathrm{ST}})=\Delta^k(\mathcal{C})\times_{\partial \Delta^k(\mathcal{C})}\partial \Delta^k(\mathcal{C}^J_{\mathrm{ST}}).
\]
    \end{itemize}
   \end{definition}

\begin{example}
    Let $G$ be a finite group, and let $\alpha$ be a $\K^\ast$-valued 2-cocycle on $G$, and $J_\alpha$ be the composition of $\alpha$ with the inclusion $\mathbf{B}^2\K^\ast\hookrightarrow 2\Vect_K$. Then there is a natural equivalence $\mathbf{B}M_\alpha\cong (\mathbf{B}G)^{J_\alpha}_{\mathrm{ST}}$. 
\end{example}

\begin{proposition}\label{prop:st-category}
    Let $J\colon \mathcal{C}\to 2\Vect_\K$ be a functor. The Stolz-Teichner category $\mathcal{C}^J_{\mathrm{ST}}$ of $\mathcal{C}^J$ is (naturally equivalent to) the full subcategory of $\mathcal{C}^J$ on the objects $(X_i,J(X_i))$, where $J(X_i)$ is seen as a right $J(X_i)$-module.
\end{proposition}
\begin{proof}
A 1-morphism from $(X_i,J(X_i))$ to $(X_j,J(X_j))$ in  $\mathcal{C}^J$ is a pair $(f_{ij},\varphi_{f_{ij}})$ consisting of a 1-simplex $X_i\xrightarrow{f_{ij}} X_j$ in $\mathcal{C}$ and a morphism of right $J(X_i)$-modules 
$
\varphi_{f_{ij}}\colon J(X_i) \to J(X_j)\otimes_{J(X_j)}J(f_{ij})=J(f_{ij})
$. 
Via the canonical isomorphism of left $J(X_j)$-modules
\begin{align*}
    \mathrm{Hom}_{\mathbf{Mod}_{J(X_i)}}(J(X_i),J(f_{ij}))&\xrightarrow{\sim} J(f_{ij})\\
    \varphi&\mapsto \varphi(1),
\end{align*}
the morphism $\varphi_{f_{ij}}$ is equivalently an element $v_{f_{ij}}$ of $J(f_{ij})$, i.e., a pointing of $J(f_{ij})$. In other words, 1-simplices in  $\mathcal{C}^J$ with vertices $(X_i,J(X_i))$ and $(X_j,J(X_j))$ are exactly the 1-simplices in  $\mathcal{C}^J_{\mathrm{ST}}$ with vertices $X_i$ and $X_j$. Passing to 2-simplices, the commutativity of the diagram of morphisms of right $J(X_i)$-modules
\begin{equation*}
\begin{tikzcd}[column sep=large]
            J(X_i) \arrow[d,"{\varphi_{f_{ij}}}"']\arrow[r,"\varphi_{f_{ik}}"]&J(f_{ik})
            \\
            J(f_{ij})\arrow[r,"{\varphi_{f_{jk}}\otimes\mathrm{id}}"]&J(f_{jk})\otimes_{J(X_j)}J(f_{ij})\arrow[u," J(\Xi_{ijk})"']
            \end{tikzcd}
\end{equation*}
is equivalent to the single equation
$
\varphi_{f_{ij}}(1)=J(\Xi_{ijk})(\varphi_{f_{jk}}(1)\otimes \varphi_{f_{ij}}(1)),
$
i.e., to the single equation
$
v_{f_{ij}}=J(\Xi_{ijk})(v_{f_{jk}}\otimes v_{f_{ij}}).
$
Since the higher simplices in  $\mathcal{C}^J$ and in $\mathcal{C}^J_{\mathrm{ST}}$ are trivial, this concludes the proof. 
\end{proof}

\section{Linear representations of $\mathcal{C}^J_{\mathrm{ST}}$
}
{
By restricting along the full inclusion $\mathcal{C}^{J}_{\mathrm{ST}}\hookrightarrow \mathcal{C}^{J}$, 
linear representations of the category $\mathcal{C}^J$ induce linear representations of of the category $\mathcal{C}^{J}_{\mathrm{ST}}$. In this section we are going to characterize
 those 
 linear representations of  $\mathcal{C}^{J}_{\mathrm{ST}}$ that are obtained this way.}
As a first step we will define a notion of ``$J$ acting as scalars'' for a linear representation of ${\mathcal{C}}^{J}_{\mathrm{ST}}$ and will show that restriction to ${\mathcal{C}}^{J}_{\mathrm{ST}}$ induces a morphism
\[
\mathrm{Hom}_{\mathrm{Vect}_\K}(\mathcal{C}^{J},\mathrm{Vect}_\K)\to \{F\in \mathrm{Hom}(\mathcal{C}^{J}_{\mathrm{ST}},\mathrm{Vect}_\K) |\, J\text{ acts as scalars}\}.
\]
Next we will show that this is indeed an equivalence. This fact can be seen as a multi-object version of the Eilenberg--Watts theorem.

\begin{lemma}\label{lemma:map-of-monoids}
Let $F\colon \mathcal{C}^J_{\mathrm{ST}}\to \Vect_\K$ be a functor. Then for any object $X$ in $\mathcal{C}$ the map
\begin{align*}
    \lambda_X\colon J(X)&\to \mathrm{End}_\K(F(X))\\
    j&\mapsto F(\mathrm{id}_X,j)
\end{align*}
is a map of monoids.
\end{lemma}
\begin{proof}
We begin by showing that $\lambda_X$ is indeed well defined. Since $J(\mathrm{id}_X)$ is the $(J(X),J(X))$ bimodule $J(X)$, an endomorphism of $X$ in $\mathcal{C}^J_{\mathrm{ST}}$ covering $\mathrm{id}_X\colon X\to X$ is a pointing of $J(X)$. Therefore we have that
$
X\xrightarrow{(\mathrm{id}_X,j)}X
$
 is a morphism in $\mathcal{C}^J_{\mathrm{ST}}$ for any $j\in J(X)$, and so 
$
F(X)\xrightarrow{F(\mathrm{id}_X,j)}F(X)
$
is a morphism in $\Vect_\K$. 
Since the natural isomorphism $J(X)\otimes_{J(X)}J(X)\to J(X)$ corresponding to the identity $\mathrm{id}_X\circ\mathrm{id}_X=\mathrm{id}_X$ is the multiplication of $J(X)$, the
 2-simplex 
\[
 \begin{tikzcd}
            & {X} \arrow[dr,"{\mathrm{id}_X}"]&\\
            {X}\arrow[ru,"{\mathrm{id}_X}"]\arrow[rr,"{\mathrm{id}_X}"']&\arrow[u,Leftarrow,shorten <=2mm, shorten >=2mm,"{\mathrm{id}}"']{}&{X}
            \end{tikzcd}.
\]
of $\mathcal{C}$ and the triple of 1-simplices 
 \[
     X\xrightarrow{({\mathrm{id}_X},j_1)}X, \qquad\qquad\qquad
     \\
     X\xrightarrow{({\mathrm{id}_X},j_2)}X, \qquad\qquad\qquad
     \\
     X\xrightarrow{({\mathrm{id}_X},j_1j_2)}X
 \]
 in $\mathcal{C}^J_{\mathrm{ST}}$ are such that $
 J(\mathrm{id})(j_1\otimes j_2)=j_1 j_2$.
So we have the 
2-simplex
        \[
 \begin{tikzcd}
            & {X} \arrow[dr,"{(\mathrm{id}_X,j_1)}"]&\\
            {X}\arrow[ru,"{(\mathrm{id}_X,j_2)}"]\arrow[rr,"{(\mathrm{id}_X,j_1j_2)}"']&\arrow[u,Leftarrow,shorten <=2mm, shorten >=2mm,"\mathrm{id}"']{}&{X}
            \end{tikzcd}.
\]
in $\mathcal{C}^J_{\mathrm{ST}}$. Applying $F$ to this 2-simplex we find the identity 
$
F(\mathrm{id}_X,j_1j_2)=F(\mathrm{id}_X,j_1)\circ F(\mathrm{id}_X,j_2)$,
i.e., the identity
$\lambda_X(j_1j_2)=\lambda_X(j_1)\circ \lambda_X(j_2)$.
\end{proof}
\begin{definition}
Let $F\colon \mathcal{C}^J_{\mathrm{ST}}\to \Vect_\K$ be a functor. We say that $F$ is such that $J$ \emph{acts as scalars} if the map $\lambda_X$ from Lemma \ref{lemma:map-of-monoids} is a map of $\K$-algebras, for any object $X$ of $\mathcal{C}$. Equivalently, $J$ acts as scalars if $\lambda_X$ makes $F(X)$ a left $J(X)$-module, for any object $X$ of $\mathcal{C}$.
\end{definition}

\begin{lemma}
Let   $E\colon \mathcal{C}^J\to \Vect_\K$ be a $\Vect_\K$-linear functor which is  additive and cocontinuous over $\mathcal{C}$, and let $E_{\mathrm{ST}}\colon \mathcal{C}^J_{\mathrm{ST}}\to \Vect_\K$ be the restriction of $E$ to $\mathcal{C}^J_{\mathrm{ST}}$. Then $E_{\mathrm{ST}}$ is such that $J$ acts as scalars.  
\end{lemma}
\begin{proof}
Let $X$ be an object of $\mathcal{C}$. Since $E\bigr\vert_X\colon \mathrm{Mod}_{J(X)}\to \Vect_\K$ is  additive and cocontinuous, by the Eilenberg--Watts theorem, we have that 
$
E\bigr\vert_X\cong -\otimes_{J(X)}V_X
$
for some left $J(X)$-module $V_X$. We therefore have 
$
E_{\mathrm{ST}}(X)=E(X,J(X))=E(\iota_X(J(X))=E\bigr\vert_X(J(X))=V_X,
$
and $E_{\mathrm{ST}}(\mathrm{id}_X,j)$ is the morphism of $\K$-vector spaces given by
\begin{align*}
V_X\cong J(X)\otimes_{J(X)}V_X&\xrightarrow{j\otimes \mathrm{id}_{V_X}}J(X)\otimes_{J(X)}V_X\cong V_X\\
v\mapsto 1\otimes v&\mapsto j\otimes v\mapsto j\cdot v,
\end{align*}
where we used that left multiplication by an element $j\in J(X)$ is naturally a morphism of right $J(X)$-modules from $J(X)$ to itself.
The map $\lambda_X$ from Lemma \ref{lemma:map-of-monoids} is then $
j\mapsto j\cdot-$ 
and this is a map of $\K$-algebras from $J(X)$ to $\mathrm{End}_\K(V_X)$.
\end{proof}
Summing up, writing
$\mathrm{Hom}_{\K}(\mathcal{C}^{J}_{\mathrm{ST}},\mathrm{Vect}_\K)$
to denote the subset of $\mathrm{Hom}(\mathcal{C}^{J}_{\mathrm{ST}},\mathrm{Vect}_\K)$ consisting of those linear representations of $\mathcal{C}^{J}_{\mathrm{ST}}$ such that {$J$} acts as scalars, we have proved that the restriction along the full embedding $\mathcal{C}^{J}_{\mathrm{ST}}\hookrightarrow \mathcal{C}^{J}$ induces a morphism $
\mathrm{Hom}_{\mathrm{Vect}_\K}(\mathcal{C}^{J},\mathrm{Vect}_\K)\to \mathrm{Hom}_{\K}(\mathcal{C}^{J}_{\mathrm{ST}},\mathrm{Vect}_\K)$.

\subsection{A multiobject  Eilenberg-Watts type theorem}\label{sec:more-general-EW}
By Proposition \ref{prop:anom-to-lin} we have a distinguished map from
anomalous representations of $\mathcal{C}$ with anomaly $J$
to $\mathrm{Hom}_{\mathrm{Vect}_\K}(\mathcal{C}^{J},\mathrm{Vect}_\K)$
and by the results in the previous section we have a restriction functor
$
\mathrm{Hom}_{\mathrm{Vect}_\K}(\mathcal{C}^{J},\mathrm{Vect}_\K)\to \mathrm{Hom}_{\K}(\mathcal{C}^{J}_{\mathrm{ST}},\mathrm{Vect}_\K)$.
In this section we will show that there is also a natural functor bringing us back, from
$
\mathrm{Hom}_{\K}(\mathcal{C}^{J}_{\mathrm{ST}},\mathrm{Vect}_\K)$ to anomalous representations of $\mathcal{C}$ with anomaly $J$,
establishing a commuting triple of equivalences
\begin{equation*}
            \begin{tikzcd}[column sep=-2em]
               {\mathrm{Hom}_{\mathrm{Vect}_\K}(\mathcal{C}^{J},\mathrm{Vect}_\K)}\arrow[rr]\arrow[dr]&& {\mathrm{Hom}_{\K}(\mathcal{C}^{J}_{\mathrm{ST}},\mathrm{Vect}_\K)}\arrow[dl]\arrow[ll]\\
                {}&\{\text{anomalous representations of $\mathcal{C}$ with anomaly $J$}\}\arrow[ur]\arrow[ul].&{}
            \end{tikzcd}
        \end{equation*}
{The proof of this result, given as Theorem \ref{thm:equivalences} below, will take the whole remainder of this section and will suggest regarding the above commutative triple of equivalences as a multi-object  analogue of the Eilenberg--Watts theorem on additive and cocontinuous functors
between categories of modules. }

\begin{lemma}\label{lemma:ZF}
 Let $J\colon \mathcal{C}\to 2\Vect_\K$ be a functor and let $F\in \mathrm{Hom}_{\K}(\mathcal{C}^{J}_{\mathrm{ST}},\mathrm{Vect}_\K)$. Then there is an anomalous representation $Z_F$ of $\mathcal{C}$ with anomaly $J$,  
 \begin{equation*}
            \begin{tikzcd}
            \mathcal{C}
                \arrow[rr, "J"]\arrow[dr]&{}&
                2\mathrm{Vect}_\K\\
                {}&\ast\arrow[u,shorten <=3pt,shorten > =1pt,Rightarrow,"Z_{F}"',pos=.5]\ar[ur]&{}\\
            \end{tikzcd},
        \end{equation*}
with $Z_F(X)=F(X)$, for any object $X$ in $X$.        
\end{lemma}
\begin{proof}
   By Remark \ref{rem:Z-explicit}, to define $Z_F$ we need to provide
 \begin{itemize}
   \item a left $J(X_i)$-module $Z_F(X_i)$, thought of as a $(J(X_i),\K)$-bimodule, for every object (0-simplex) $X_i$ of $\mathcal{C}$;  
   \item  a morphism of left $J(X_j)$-modules 
\[
Z_F(f_{ij})\colon J(f_{ij})\otimes_{J(X_i)} Z_F(X_i)\to Z_F(X_j),
\]
for any morphism $f_{ij}\colon X_i\to X_j$ in $\mathcal{C}$;
\end{itemize}
such that the diagram of morphisms of left $J(X_k)$-modules
\begin{equation}\label{eq:diagram:Z-J-F}
\begin{tikzcd}[column sep=large]
            J(f_{jk})\otimes_{J(X_j)}J(f_{ij})\otimes Z_F(X_i) \arrow[d,"J(\Xi_{ijk})\otimes{\mathrm{id}}"']\arrow[r,"\mathrm{id}\otimes Z_F(f_{ij})"]&J(f_{jk})\otimes_{J(X_j)}Z_F(X_j)
            \arrow[d,"Z_F(f_{jk})"]\\
            J(f_{ik})\otimes_{J(X_i)} Z_F(X_i)\arrow[r,"Z_F(f_{ik})"]&Z(X_k)
            \end{tikzcd}
\end{equation}
commutes, for any  2-simplex 
\[
 \begin{tikzcd}
            & {X_j} \arrow[dr,"{f_{jk}}"]&\\
            {X_i}\arrow[ru,"{f_{ij}}"]\arrow[rr,"{f_{ik}}"']&\arrow[u,Leftarrow,shorten <=2mm, shorten >=2mm,"\Xi_{ijk}"']{}&{X_k}
            \end{tikzcd}.
\]
of $\mathcal{C}$.
We set $Z_F(X_i)=F(X_i)$, with the left $J(X_i)$-module structure given by the fact that $F$ is an element in  $\mathrm{Hom}_{\K}(\mathcal{C}^{J}_{\mathrm{ST}},\mathrm{Vect}_\K)$. Next, to define the morphism of left $J(X_j)$-modules 
$
Z_F(f_{ij})\colon J(f_{ij})\otimes_{J(X_i)} F(X_i)\to F(X_j)
$.
we notice that, 
for any fixed $f_{ij}\colon X_i\to X_j$, for any element $j_{f_{ij}}\in J(f_{ij})$ we have a morphism $(f_{ij},j_{f_{ij}})$ from $X_i$ to $X_j$ in $\mathcal{C}^J_{\mathrm{ST}}$. Then we define
$
Z_F(f_{ij})(j_{f_{ij}}\otimes v_i)=F(f_{ij},j_{f_{ij}})(v_i)$.
We have to check that this is well defined, i.e., that,
\begin{equation}\label{eq:first-to-check}
Z_F(f_{ij})(j_{f_{ij}}\cdot j_{X_i}\otimes v_i)=Z_F(f_{ij})(j_{f_{ij}} \otimes j_{X_i}\cdot v_i)
\end{equation}
for any $j_{X_i}\in J(X_i)$, and that $Z_F(f_{ij})$ is a morphism of left $J(X_j)$-modules, i.e., 
\begin{equation}\label{eq:second-to-check}
Z_F(f_{ij})(j_{X_j}\cdot j_{f_{ij}}\otimes v_i)=j_{X_j}\cdot Z_F(f_{ij})(j_{f_{ij}} \otimes  v_i),
\end{equation}
for any $j_{X_j}\in J(X_j)$.
To prove \eqref{eq:first-to-check}, we use the definition of the $J(X_i)$-module structure on $F(X_i)$  to compute
\begin{align*}
  Z_F(f_{ij})(j_{f_{ij}} \otimes j_{X_i}\cdot v_i)&= F(f_{ij},j_{f_{ij}})(j_{X_i}\cdot v_i)\\
  &=(F(f_{ij},j_{f_{ij}})\circ F(\mathrm{id}_{X_i},j_{X_i}))(v_i)\\
  &=F(f_{ij},j_{f_{ij}}\cdot j_{X_i} )(v_i)\\
  &=Z_F(f_{ij})(j_{f_{ij}}\cdot j_{X_i}\otimes v_i),
\end{align*}
where in the next to last step we used that we have the 2-simplex
 \[
 \begin{tikzcd}
            & {X_i} \arrow[dr,"{(f_{ij},j_{f_{ij}}})"]&\\
            {X_i}\arrow[ru,"{(\mathrm{id}_{X_i},j_{X_j})}"]\arrow[rr,"{(f_{ij},j_{X_j}\cdot j_{f_{ij}})}"']&\arrow[u,Leftarrow,shorten <=2mm, shorten >=2mm,"\mathrm{id}"']{}&{X_j}
            \end{tikzcd}.
\]
in $\mathcal{C}^J_{\mathrm{ST}}$.
Similarly, to prove \eqref{eq:second-to-check}, we use the definition of the $J(X_j)$-module structure on $F(X_j)$ to compute
\begin{align*}
j_{X_j}\cdot Z_F(f_{ij})(j_{f_{ij}} \otimes  v_i)&=F(\mathrm{id}_{X_j},j_{X_j})Z_F(f_{ij})(j_{f_{ij}} \otimes  v_i)\\
&=(F(\mathrm{id}_{X_j},j_{X_j})\circ F(f_{ij},j_{f_{ij}}))(v_i)\\
&=F(f_{ij},j_{x_j}\cdot j_{f_{ij}})(v_i)\\
&=Z_F(f_{ij})(j_{X_j}\cdot j_{f_{ij}}\otimes v_i),
\end{align*}
where in the next to last step we used that we have the 2-simplex
 \[
 \begin{tikzcd}
            & {X_j} \arrow[dr,"{(\mathrm{id}_{X_j},j_{X_j})}"]&\\
            {X_i}\arrow[ru,"{(f_{ij},j_{f_{ij}})}"]\arrow[rr,"{(f_{ij},j_{X_j}\cdot j_{f_{ij}})}"']&\arrow[u,Leftarrow,shorten <=2mm, shorten >=2mm,"\mathrm{id}"']{}&{X_j}
            \end{tikzcd}.
\]
in $\mathcal{C}^J_{\mathrm{ST}}$. Finally, we verify that the diagram \eqref{eq:diagram:Z-J-F} commutes. We have
\begin{align*}
    (Z_F(f_{jk})\circ (\mathrm{id}_{J(f_{jk})}\otimes Z_F(f_{ij})))(j_{f_{jk}}\otimes j_{f_{ij}}\otimes v_i)&=Z_F(f_{jk}(j_{f_{jk}}\otimes Z_F(f_{ij})))(j_{f_{ij}}\otimes v_i)\\
    &=Z_F(f_{jk}(j_{f_{jk}}\otimes F(f_{ij},j_{f_{ij}})(v_i))\\
    &=F(f_{jk},j_{f_{jk}})(F(f_{ij},j_{f_{ij}})(v_i)))\\
    &=(F(f_{jk},j_{f_{jk}})\circ F(f_{ij},j_{f_{ij}}))(v_i).
\end{align*}
Let $j_{f_{ik}}$ be the element of $J(f_{ik})$ defined by $j_{f_{ik}}=J(\Xi_{ijk})(j_{f_{jk}}\otimes j_{f_{ij}})$. Then 
\[
 \begin{tikzcd}
            & {X_j} \arrow[dr,"{(f_{jk},j_{f_{jk}})}"]&\\
            {X_i}\arrow[ru,"{(f_{ij},j_{f_{ij}})}"]\arrow[rr,"{(f_{ik},j_{f_{ik}})}"']&\arrow[u,Leftarrow,shorten <=2mm, shorten >=2mm,"\Xi_{ijk}"']{}&{X_k}
            \end{tikzcd}
\]
is a 2-simplex in $\mathcal{C}^J_{\mathrm{ST}}$ and so
\[
 \begin{tikzcd}
            & {F(X_j)} \arrow[dr,"{F(f_{jk},j_{f_{jk}})}"]&\\
            {F(X_i)}\arrow[ru,"{F(f_{ij},j_{f_{ij}})}"]\arrow[rr,"{F(f_{ik},j_{f_{ik}})}"']&{}&{F(X_k)}
            \end{tikzcd}
\]
is a 2-simplex in $\Vect_\K$, i.e., 
$
F(f_{jk},j_{f_{jk}})\circ F(f_{ij},j_{f_{ij}})=F(f_{ik},j_{f_{ik}})
$. 
We then conclude by computing
\begin{align*}
    (Z_F(f_{ik})\circ (J(\Xi_{ijk})\otimes \mathrm{id}){Z_F(X_i)}))(j_{f_{jk}}\otimes j_{f_{ij}}\otimes v_i)&=Z_F(f_{ik})(J(\Xi_{ijk})(j_{f_{jk}}\otimes j_{f_{ij}})\otimes v_i)\\
    &=Z_F(f_{ik})(j_{f_{ik}}\otimes v_i)\\
    &=F(f_{ik},j_{f_{ik}})(v_i).
\end{align*}
\end{proof}

\begin{proposition}
The correspondence
\begin{align*}
\mathrm{Hom}_{\K}(\mathcal{C}^{J}_{\mathrm{ST}},\mathrm{Vect}_\K)&\to 
\{\text{anomalous representations of $\mathcal{C}$ with anomaly $J$}\}\\
F&\mapsto Z_F
\end{align*}
from Lemma \ref{lemma:ZF} is the inverse of the composition
\begin{align*}
\{\text{anom. representations of $\mathcal{C}$ with anomaly $J$}\}
&\to
\mathrm{Hom}_{\mathrm{Vect}_\K}(\mathcal{C}^{J},\mathrm{Vect}_\K)
\to \mathrm{Hom}_{\K}(\mathcal{C}^{J}_{\mathrm{ST}},\mathrm{Vect}_\K)\\
Z&\mapsto E_Z\mapsto (E_Z)_{\mathrm{ST}},
\end{align*}
where $Z\mapsto E_Z$ is the construction from Proposition \ref{prop:anom-to-lin}, and $(-)_{\mathrm{ST}}$ is the restriction to the full subcategory $\mathcal{C}^J_{\mathrm{ST}}$ of $\mathcal{C}^J$. In other words, the triangle
\begin{equation}\label{eq:triangle}
            \begin{tikzcd}[column sep=-3.5em]
               {\mathrm{Hom}_{\mathrm{Vect}_\K}(\mathcal{C}^{J},\mathrm{Vect}_\K)}\arrow[rr]&& {\mathrm{Hom}_{\K}(\mathcal{C}^{J}_{\mathrm{ST}},\mathrm{Vect}_\K)}\arrow[dl]\\
                {}&\{\text{anomalous representations of $\mathcal{C}$ with anomaly $J$}\}\arrow[ul].&{}
            \end{tikzcd}
        \end{equation}
is the identity on $\mathrm{Hom}_{\K}(\mathcal{C}^{J}_{\mathrm{ST}},\Vect_\K)$.       
\end{proposition}

\begin{proof}
Let $X_i$ be an object in $\mathcal{C}^J_{\mathrm{ST}}$. Then, by Proposition \ref{prop:anom-to-lin},
\[
(E_{Z_F})_{\mathrm{ST}}(X_i)=E_{Z_F}(X_i,J(X_i))=J(X_i)\otimes_{J(X_i)}Z_F(X_i)=F(X_i),
\]
where in the last step we used that, by Lemma \ref{lemma:ZF},  $Z_F(X_i)$ is the $\K$-vector space $F(X_i)$ endowed with a certain left $J(X_i)$-module structure. Let now $(f_{ij},j_{f_{ij}})\colon X_i\to X_j$ be a morphism in $\mathcal{C}^J_{\mathrm{ST}}$. We have $(E_{Z_F})_{\mathrm{ST}}(f_{ij},j_{f_{ij}})=E_{Z_F}(f_{ij},\varphi_{j_{f_{ij}}})$, where $\varphi_w\colon J(X_i)\to J(X_j)\otimes_{J(X_j)}J(f_{ij})=J(f_{ij})$ is the unique morphism of left $J(X_i)$ modules with $\varphi_w(1)=w$.
By Proposition \ref{prop:anom-to-lin}, the morphism $E_{Z_F}(f_{ij},\varphi_{j_{f_{ij}}})$ is the composition
\begin{align*}
J(X_i)\otimes_{J(X_i)} Z_F(X_i) \xrightarrow{\varphi_{j_{f_{ij}}}\otimes \mathrm{id}_{Z_F(X_i)}}&
J(X_j)\otimes_{J(X_j)}J(f_{ij})\otimes_{J(X_i)} Z_F(X_i)\\
&\qquad\xrightarrow{\mathrm{id}_{J(X_j)}\otimes Z_F(f_{ij})} J(X_j)\otimes_{J(X_j)}Z_F(X_j).
\end{align*}
The source and target of this morphism are canonically identified with $F(X_i)$ and $F)X_j)$, respectively. With this identification, the image of an element $v$ in $F(X_i)$ under $(E_{Z_F})_{\mathrm{ST}}(f_{ij},j_{f_{ij}})$ is given by
\[
Z_F(f_{ij})(j_{f_{ij}}\otimes v)=F(f_{ij},j_{f_{ij}})(v),
\]
by definition of $Z_F$ on morphisms (see the proof of Lemma \ref{lemma:ZF}). This shows that $(E_{Z_F})_{\mathrm{ST}}(f_{ij},j_{f_{ij}})$ coincides with $F$ at the level of 1-morphisms, too. Since the target $\Vect_\K$ of the functors $(E_{Z_F})_{\mathrm{ST}}(f_{ij},j_{f_{ij}})$ and $F$ is a 1-category, this concludes the proof.
\end{proof}

\begin{proposition}
    The triangle \eqref{eq:triangle}  is the identity on $\mathrm{Hom}_{\Vect_\K}(\mathcal{C}^{J},\Vect_\K)$.
\end{proposition}
\begin{proof}
Let $E\in \mathrm{Hom}_{\Vect_\K}(\mathcal{C}^{J},\Vect_\K)$, and let $(X_i,L_{X_i})$ be an object in $\mathcal{C}^J$. By Proposition \ref{prop:anom-to-lin}, Lemma \ref{lemma:ZF} and Proposition \ref{prop:st-category}, we have
\[
E_{Z_{E_{\mathrm{ST}}}}(X_i,L_i)=L_{X_i}\otimes_{J(X_i)}Z_{E_{\mathrm{ST}}}(X_i)=
L_{X_i}\otimes_{J(X_i)}E_{\mathrm{ST}}(X_i)=L_{X_i}\otimes_{J(X_i)}E(X_i,J(X_i)).
\]
On the other hand, $
E(X_i,L_{X_i})=E\bigr\vert_{X_i}(L_{X_i})$.
Since $E\bigr\vert_{X_i}$ is additive and cocontinuous by definition of $\mathrm{Hom}_{\Vect_\K}(\mathcal{C}^{J},\Vect_\K)$, by the Eilenberg-Watts theorem we have
$
E\bigr\vert_{X_i}\cong -\otimes_{J(X_i)} E\bigr\vert_{X_i}(J(X_i))$,
so that we have a natural isomorphism
\[
E(X_i,L_{X_i})\cong L_{X_i}\otimes_{J(X_i)} E\bigr\vert_{X_i}(J(X_i))=
L_{X_i}\otimes_{J(X_i)} E(X_i,J(X_i))=E_{Z_{E_{\mathrm{ST}}}}(X_i,L_i).
\]
More precisely, the isomorphism $L_{X_i}\otimes_{J(X_i)} E(X_i,J(X_i))\xrightarrow{\sim}E(X_i,L_{X_i})$ is given as follows:
\begin{align*}
    L_{X_i}\otimes_{J(X_i)}E(X_i,J(X_i))&\xrightarrow{E(\mathrm{id}_{X_i},\varphi_{-})}E(X_i,L_{X_i})\\
    l_{i}\otimes v_i&\mapsto E(\mathrm{id}_{X_i},\varphi_{l_i})(v_i),
\end{align*}
where $\varphi_{l_i}\colon J(X_i)\to L_{X_i}$ is the unique morphism of right $J(X_i)$-modules with $\varphi_{l_i}(1)=l_i$.
\end{proof}

 Let now $(f_{ij},\varphi_{f_{ij}})\colon (X_i,L_{X_i})\to (X_j,L_{X_j})$ be a morphism in $\mathcal{C}^J$. Then $E_{Z_{E_{\mathrm{ST}}}}(f_{ij},\varphi_{f_{ij}})$ is the composition
\begin{align*}
L_{X_i}\otimes_{J(X_i)} Z_{E_{\mathrm{ST}}}(X_i) &\xrightarrow{\varphi_{f_{ij}}\otimes \mathrm{id}_{Z_{E_{\mathrm{ST}}}(X_i)}}
L_{X_j}\otimes_{J(X_j)}J(f_{ij})\otimes_{J(X_i)} Z_{E_{\mathrm{ST}}}(X_i)\\
&\xrightarrow{\mathrm{id}_{L_{X_j}}\otimes Z_{E_{\mathrm{ST}}}(f_{ij})} L_{X_j}\otimes_{J(X_j)}Z_{E_{\mathrm{ST}}}(X_j),
\end{align*}
i.e., the composition
\begin{align*}
L_{X_i}\otimes_{J(X_i)} E(X_i,J(X_i)) &\xrightarrow{\varphi_{f_{ij}}\otimes \mathrm{id}_{E(X_i,J(X_i))}}
L_{X_j}\otimes_{J(X_j)}J(f_{ij})\otimes_{J(X_i)} E(X_i,J(X_i))\\
&\xrightarrow{\mathrm{id}_{L_{X_j}}\otimes Z_{E_{\mathrm{ST}}}(f_{ij})} L_{X_j}\otimes_{J(X_j)}E(X_j,J(X_j)).
\end{align*}
 Now we use the Eilenberg-Watts theorem again. The naturality of the isomorphism $
E\bigr\vert_{X_i}\cong -\otimes_{J(X_i)} E\bigr\vert_{X_i}(J(X_i))=-\otimes_{J(X_i)} E(X_i, J(X_i))$ gives us the commutative diagram
\[
\begin{tikzcd}
L_{X_i}\otimes_{J(X_i)} E(X_i,J(X_i)) \ar[d,"\wr"'] \ar[rrr,"\varphi_{f_{ij}}\otimes \mathrm{id}_{E(X_i,J(X_i))}"]&&&
L_{X_j}\otimes_{J(X_j)}J(f_{ij})\otimes_{J(X_i)} E(X_i,J(X_i))\ar[d,"\wr"]\\
{E\bigr\vert_{X_i}(L_{X_i})}\ar[rrr,"E\vert_{X_i}(\varphi_{f_{ij}})"]&&&{E\bigr\vert_{X_i}(L_{X_j}\otimes_{J(X_j)}J(f_{ij}))}
\end{tikzcd},
\]
i.e., the commutative diagram
\[
\begin{tikzcd}
L_{X_i}\otimes_{J(X_i)} E(X_i,J(X_i)) \ar[d,"\wr"'] \ar[rrr,"\varphi_{f_{ij}}\otimes \mathrm{id}_{E(X_i,J(X_i))}"]&&&
L_{X_j}\otimes_{J(X_j)}J(f_{ij})\otimes_{J(X_i)} E(X_i,J(X_i))\ar[d,"\wr"]\\
{E(X_i,L_{X_i})}\ar[rrr,"{E(\mathrm{id}_{X_i},\varphi_{f_{ij}})}"]&&&
{E(X_i,L_{X_j}\otimes_{J(X_j)}J(f_{ij}))}
\end{tikzcd}
\]
For every right $J(X_j)$-module $L_{X_j}$,
we have a morphism
\[
(X_i,L_{X_j}\otimes_{J(X_j)}J(f_{ij}))\xrightarrow{(f_{ij},\mathrm{id}_{L_{X_j}\otimes_{J(X_j)}J(f_{ij})})}(X_j,L_{X_j})
\]
in $\mathcal{C}^J$ and so a morphism
\[
E(X_i,L_{X_j}\otimes_{J(X_j)}J(f_{ij}))\xrightarrow{E(f_{ij},\mathrm{id}_{L_{X_j}\otimes_{J(X_j)}J(f_{ij})})}E(X_j,L_{X_j})
\]
in $\Vect_\K$. This gives linear natural transformation between cocontinuous additive functors 
     \begin{equation*}    
            \begin{tikzcd}
                {}_{\K}\mathrm{Mod}_{J(X_j)} \arrow[rr,"E\vert_{X_i}\circ (-\otimes_{J(X_j)}J(f_{ij})", bend left=30, ""{name=U}]\arrow[rr, "E\vert_{X_j}"', bend right=30, ""{name=D}]&& \phantom{mm}{}_{\K}\mathrm{Mod}_{\K}=\Vect_\K\arrow[Rightarrow,shorten >=3mm, shorten <=3mm, from=U, to=D,"E(\mathrm{id}_-)"],
            \end{tikzcd}
        \end{equation*}
where $E(\mathrm{id}_{L_{X_j}})$ is a shorthand notation for $E(f_{ij},\mathrm{id}_{L_{X_j}\otimes_{J(X_j)}J(f_{ij})})$. By the Eilenberg--Watts theorem we then have a commutative diagram
\[
\begin{tikzcd}[column sep=.7em]
L_{X_j}\otimes_{J(X_j)}J(f_{ij})\otimes_{J(X_i)}E(X_i,J(X_i))\ar[r,"{\mathrm{id}_{L_{X_j}}\otimes E(\mathrm{id}_{X_i},\varphi_-)}"]\ar[d,"\wr"]& L_{X_j}\otimes_{J(X_j)}E(X_i,J(f_{ij}))\ar[d,"\wr"]\ar[rr,"{\mathrm{id}_{L_{X_j}}\otimes E(f_{ij},\mathrm{id}_{J(f_{ij})})}"]&&
L_{X_j}\otimes_{J(X_j)}E(X_j,J(X_j))\ar[d,"\wr"]
\\
E\bigr\vert_{X_i}(L_{X_j}\otimes_{J(X_j)}J(f_{ij}))\ar[r,equal]&(E\bigr\vert_{X_i}\circ (-\otimes_{J(X_j)}J(f_{ij}))(L_{X_j})\ar[rr,"{E(f_{ij},\mathrm{id}_{L_{X_j}\otimes_{J(X_j)}J(f_{ij})})}"]&& E\bigr\vert_{X_j}(L_{X_j})
\end{tikzcd},
\]
and so a commutative diagram
\[
\begin{tikzcd}
L_{X_j}\otimes_{J(X_j)}J(f_{ij})\otimes_{J(X_i)}E(X_i,J(X_i))\ar[rrr,"{\mathrm{id}_{L_{X_j}}\otimes E(f_{ij},\varphi_-)}"]\ar[d,"\wr"]&&& 
L_{X_j}\otimes_{J(X_j)}E(X_j,J(X_j))\ar[d,"\wr"]
\\
E(X_i,L_{X_j}\otimes_{J(X_j)}J(f_{ij}))\ar[rrr,"{E(f_{ij},\mathrm{id}_{L_{X_j}\otimes_{J(X_j)}J(f_{ij})})}"]&&& E(X_j,L_{X_j})
\end{tikzcd},
\]
where in the top horizontal arrow we used that for any $j_{f_{ij}}$ in $J(f_{ij})$ we have  the 2-simplex
\[
 \begin{tikzcd}
            & {(X_i,J(f_{ij}))} \arrow[dr,"{(f_{ij},\mathrm{id}_{J(f_{ij})})}"]&\\
            {(X_i,J(X_i))}\arrow[ru,"{(\mathrm{id}_{X_i},\varphi_{j_{f_{ij}}})}"]\arrow[rr,"{(f_{ij},\varphi_{j_{f_{ij}}})}"']&\arrow[u,Leftarrow,shorten <=2mm, shorten >=2mm,"\mathrm{id}"']{}&{(X_j,J(X_j))}
            \end{tikzcd}
\]
in $\mathcal{C}^J$.
By Lemma \ref{lemma:ZF}, we have
\[
Z_{E_{\mathrm{ST}}}(f_{ij})(j_{f_{ij}}\otimes v_i)=E_{\mathrm{ST}}(f_{ij},j_{f_{ij}})(v_i)=
E(f_{ij},\varphi_{j_{f_{ij}}})(v_i),
\]
for any $j_{f_{ij}}$ in $J(f_{ij})$ and any $v_i$ in $E(X_i,J(X_i))$, that is,
\[
Z_{E_{\mathrm{ST}}}(f_{ij})=E(f_{ij},\varphi_{-})\colon J(f_{ij})\otimes_{J(X_i)}E(X_i,J(X_i))\to E(X_j,J(X_j)).
\]
Putting all the pieces together, we get the commutative diagram
\[
\hskip -1 cm
\begin{tikzcd}
    L_{X_i}\otimes_{J(X_i)} Z_{E_{\mathrm{ST}}}(X_i)\ar[d,"\wr"] \arrow[rr,"{\varphi_{f_{ij}}\otimes \mathrm{id}_{Z_{E_{\mathrm{ST}}}(X_i)}}"]&&
L_{X_j}\otimes_{J(X_j)}J(f_{ij})\otimes_{J(X_i)} Z_{E_{\mathrm{ST}}}(X_i)\ar[d,"\wr"]\arrow[rr,"{\mathrm{id}_{L_{X_j}}\otimes Z_{E_{\mathrm{ST}}}(f_{ij})}"]&& L_{X_j}\otimes_{J(X_j)}Z_{E_{\mathrm{ST}}}(X_j)\ar[d,"\wr"]\\
{E(X_i,L_{X_i})}\ar[rr,"{E(\mathrm{id}_{X_i},\varphi_{f_{ij}})}"]&& E(X_i,L_{X_j}\otimes_{J(X_j)}J(f_{ij}))\ar[rr,"{E(f_{ij},\mathrm{id}_{L_{X_j}\otimes_{J(X_j)}J(f_{ij})})}"]&& E(X_j,L_{X_j})
\end{tikzcd},
\]
and so the commutative diagram
\[
\begin{tikzcd}
    E_{Z_{E_{\mathrm{ST}}}}(X_i,L_{X_i})\ar[d,"\wr"] \arrow[rr,"{E_{Z_{E_{\mathrm{ST}}}}(f_{ij},\varphi_{f_{ij}})}"]&& E_{Z_{E_{\mathrm{ST}}}}(X_j,L_{X_j})\ar[d,"\wr"]\\
{E(X_i,L_{X_i})}\ar[rr,"{E(f_{ij},\varphi_{f_{ij}})}"]&& E(X_j,L_{X_j})
\end{tikzcd},
\]
where in the bottom horizontal arrow we used  the 2-simplex
\[
 \begin{tikzcd}
            & {(X_i,L_{X_j}\otimes_{J(X_j)}J(f_{ij}))} \arrow[dr,"{(f_{ij},\mathrm{id}_{L_{X_j}\otimes_{J(X_j)} J(f_{ij})})}"]&\\
            {(X_i,L_{X_i})}\arrow[ru,"{(\mathrm{id}_{X_i},\varphi_{{f_{ij}}})}"]\arrow[rr,"{(f_{ij},\varphi_{{f_{ij}}})}"']&\arrow[u,Leftarrow,shorten <=2mm, shorten >=2mm,"\mathrm{id}"']{}&{(X_j,L_{X_j})}
            \end{tikzcd}
\]
in $\mathcal{C}^J$, and where the vertical arrows are the isomorphisms given by the Eilenberg--Watts theorem. Hence we see that the Eilenberg--Watts theorem identifies $E_{Z_{E_{\mathrm{ST}}}}$ with $E$ both at the objects and at the 1-morphisms level. Since the target $\Vect_\K$ of the functors $E_{Z_{E_{\mathrm{ST}}}}$ and $E$ is a 1-category, this concludes the proof.

\begin{proposition}
    The triangle \eqref{eq:triangle}  is the identity on anomalous representations of $\mathcal{C}$ with anomaly $J$.
\end{proposition}
\begin{proof}
Let $X_i$ be an object of $\mathcal{C}$. Then, by Lemma \ref{lemma:ZF} and Proposition \ref{prop:anom-to-lin} we have
\[
Z_{(E_Z)_{\mathrm{ST}}}(X_i)=(E_Z)_{\mathrm{ST}}(X_i)=E_Z(X_i,J(X_i))=J(X_i)\otimes_{J(X_i)}Z(X_i)=Z(X_i).
\]
If $f_{ij}\colon X_i\to X_j$ is a morphism in $\mathcal{C}$, then by Lemma \ref{lemma:ZF} the morphism
$
Z_{(E_Z)_{\mathrm{ST}}}(f_{ij})\colon J(f_{ij})\otimes_{J(X_i)}Z_{(E_Z)_{\mathrm{ST}}}(X_i)\to Z_{(E_Z)_{\mathrm{ST}}}(X_j)
$
is given by
\[
Z_{(E_Z)_{\mathrm{ST}}}(f_{ij})(j_{f_{ij}}\otimes v_i)=(E_Z)_{\mathrm{ST}}(f_{ij},j_{f_{ij}})(v_i)=E_Z(f_{ij},\varphi_{j_{f_{ij}}})(v_i),
\]
where $\varphi_{j_{f_{ij}}}\colon J(X_i)\to J(f_{ij})$ is the unique morphism of right $J(X_i)$-modules with $\varphi_{j_{f_{ij}}}(1)=j_{f_{ij}}$. By Proposition \ref{prop:anom-to-lin}, $E_Z(f_{ij},\varphi_{j_{f_{ij}}})$ is the composition
\[
Z(X_i)=J(X_i)\otimes_{J(X_i)} Z(X_i) \xrightarrow{\varphi_{j_{f_{ij}}}\otimes \mathrm{id}_{Z(X_i)}}
J(f_{ij})\otimes_{J(X_i)} Z(X_i)\xrightarrow{\mathrm{id}_{ Z(f_{ij})}} Z(X_j)
\]
and so it is given by
$
v_i\mapsto Z(f_{ij})(j_{f_{ij}}\otimes v_i)$.
This shows $Z_{(E_Z)_{\mathrm{ST}}}(f_{ij})= Z(f_{ij})$. So we see that $Z_{(E_Z)_{\mathrm{ST}}}$ and $Z$ coincide both at the level of objects and at the level of 1-morphisms of $\mathcal{C}$. Since higher simplices in $\mathcal{C}$ do not provide additional data for an anomalous representation, but constraints for the data $\{Z(X_i),Z(f_{ij})\}$, see Remark \ref{rem:Z-explicit}, this implies $Z_{(E_Z)_{\mathrm{ST}}}=Z$.
\end{proof}

Summing up, we have { proved the following result.
 \begin{theorem}\label{thm:equivalences}
There is an explicit commuting diagram of equivalences
\begin{equation*}
            \begin{tikzcd}[column sep=-3.5em]
               {\mathrm{Hom}_{\mathrm{Vect}_\K}(\mathcal{C}^{J},\mathrm{Vect}_\K)}\arrow[rr]\arrow[dr]&& {\mathrm{Hom}_{\K}(\mathcal{C}^{J}_{\mathrm{ST}},\mathrm{Vect}_\K)}\arrow[dl]\arrow[ll]\\
                {}&\{\text{anomalous representations of $\mathcal{C}$ with anomaly $J$}\}\arrow[ur]\arrow[ul].&{}
            \end{tikzcd}
        \end{equation*}
 \end{theorem}
 }

\begin{example}
Let $G$ be a finite group, and let $\alpha$ be a $\K^\ast$-valued 2-cocycle on $G$, and $J_\alpha$ be the composition of $\alpha$ with the inclusion $\mathbf{B}^2\K^\ast\hookrightarrow 2\Vect_K$. Then {Theorem} \ref{thm:equivalences} 
recovers the classical equivalence between projective representations of $G$ with 2-cocycle $\alpha$ and linear representations of $G^\alpha$ such that $\K^\ast$ acts as scalars.
\end{example}

\appendix
\section{The twisted group algebra $\K^\alpha[G]$ as a Kan extension}
In this final section we show how one recovers the classical equivalence { between the category of
projective representations of $G$ with 2-cocycle $\alpha$ and that of
$\K^\alpha[G]$-modules,}
where $\K^\alpha[G]$ is the twisted group algebra of $G$, with twist given by the 2-cocycle $\alpha$, within the framework presented in the main body of the article.
To begin with, recall we have shown that projective representations of $G$ with 2-cocycle $\alpha$ are equivalently anomalous representations of $\mathbf{B}G$ with anomaly $J_\alpha$, i.e., lax homotopy commutative diagrams of the form
\begin{equation*}
            \begin{tikzcd}
            \mathbf{B}G
                \arrow[rr, "J_\alpha"]\arrow[dr]&{}&
                2\mathrm{Vect}_\K\\
                {}&\ast\arrow[u,shorten <=3pt,shorten > =1pt,Rightarrow,"Z"',pos=.5]\ar[ur,"{\K}"']&{}\\
            \end{tikzcd},
                   \end{equation*}
where now we are stressing that the rightmost bottom arrow picks the algebra $\K$. Forgetting the rightmost algebra we are left with a diagram of the form
\begin{equation}\label{eq:to-be-kan-extended}
            \begin{tikzcd}
            \mathbf{B}G
                \arrow[rr, "J_\alpha"]\arrow[dr]&{}&
                2\mathrm{Vect}_\K\\
                {}&\ast&{}\\
            \end{tikzcd},
                   \end{equation}
and we can look for a universal completion of it of the form                   
\begin{equation}\label{eq:right-kan}
            \begin{tikzcd}
            \mathbf{B}G
                \arrow[rr, "J_\alpha"]\arrow[dr]&{}&
                2\mathrm{Vect}_\K\\
                {}&\ast\arrow[u,shorten <=3pt,shorten > =1pt,Rightarrow,"\Theta"',pos=.5]\ar[ur,"{A}"']&{}\\
            \end{tikzcd},
                   \end{equation}
where universality means that, if a completion of the form 
               \begin{equation}\label{eq:completion-B}
            \begin{tikzcd}
            \mathbf{B}G
                \arrow[rr, "J_\alpha"]\arrow[dr]&{}&
                2\mathrm{Vect}_\K\\
                {}&\ast\arrow[u,shorten <=3pt,shorten > =1pt,Rightarrow,"N"',pos=.5]\ar[ur,"{B}"']&{}\\
            \end{tikzcd}
                   \end{equation}
is given, then \eqref{eq:completion-B} uniquely factorizes as
\begin{equation}\label{eq:kan-factorization}
            \begin{tikzcd}
            \mathbf{B}G
                \arrow[rr, "J_\alpha"]\arrow[dr]&{}&
                2\mathrm{Vect}_\K\\
                {}&\ast\arrow[u,shorten <=3pt,shorten > =1pt,Rightarrow,"\Theta"',pos=.5]
            \ar[ur,"{A}"]
            \arrow[ur,bend right=50,"B"']&{}
            \arrow[ul,shorten <=12pt,shorten > =27pt,Rightarrow,"M",pos=.3]\\
            \end{tikzcd},
                   \end{equation}
where as usual uniqueness is up to homotopies that are unique up to higher homotopies, etc. The universal diagram \eqref{eq:right-kan}, if it exists, is called the \emph{right Kan extension} of $J_\alpha$ along the terminal morphism $\mathbf{B}G\to \ast$. as usual when dealing with universal properties, the right Kan extension, if it exists, is unique.
\begin{lemma}
The right Kan extension  \eqref{eq:right-kan} exists.   
\end{lemma}
\begin{proof}
The category $\mathbf{B}G$ is small and $2\Vect_\K$ is complete. These conditions are sufficient to ensure the existence of right Kan extensions, see, e.g. \cite{maclane} for the classical case and \cite{htt} for the higher categorical version.    
\end{proof}
\begin{lemma}
Let $A$ be the right Kan extension  \eqref{eq:right-kan}. Then we have an equivalence { between the category 
of lax homotopy commutative diagrams of the form
\[\begin{tikzcd}
            \mathbf{B}G
                \arrow[rr, "J_\alpha"]\arrow[dr]&{}&
                2\mathrm{Vect}_\K\\
                {}&\ast\arrow[u,shorten <=3pt,shorten > =1pt,Rightarrow,"N"',pos=.5]\ar[ur,"{B}"']&{}\\
            \end{tikzcd}
\]
\vskip -.8 cm
\noindent
and the category of $(A,B)$-bimodules.}
\end{lemma}
\begin{proof}
 The category of natural transformations 
 \[
\begin{tikzcd}
\ast\arrow[rr,bend right=20,"A"']\arrow[rr,bend left=20,"B"]&\Uparrow_M& 2\Vect_\K
\end{tikzcd}
 \]
 is (equivalent to) the category of $(A,B)$-bimodules.
\end{proof}
\begin{lemma}
The right Kan extension $\eqref{eq:right-kan}$ is 
\begin{equation*}
            \begin{tikzcd}
            \mathbf{B}G
                \arrow[rr, "J_\alpha"]\arrow[dr]&{}&
                2\mathrm{Vect}_\K\\
                {}&\ast\arrow[u,shorten <=3pt,shorten > =1pt,Rightarrow,"{\K^\alpha[G]}"',pos=.7]\ar[ur,"{\K^\alpha[G]}"']&{}\\
            \end{tikzcd},
\end{equation*}            
where $\K^\alpha[G]$ is 
the $\alpha$-twisted group algebra , i.e., the $\K$-algebra generated by the elements $x_g$, with $g$ in the group $G$, with multiplication $x_g\cdot x_h=\alpha(g,h) x_{gh}$.
\end{lemma}
\begin{proof}
A diagram of the form  \eqref{eq:completion-B} is the datum of a pair $(B,N)$, where $B$ is a $\K$-algebra and $N$ is a right $B$-module such that
\begin{enumerate}
\item For every $g$ in $G$ we have
a morphism of right $B$-modules 
\[
\xi_g\colon N\to N\otimes_B B= N.
\]
\item The morphisms $\xi_g$ are such that
\[
\xi_g\circ \xi_h=\xi_{gh}\cdot \alpha(g,h)= \alpha(g,h)\, \xi_{gh}. 
\]
\end{enumerate}
From this we see that the pair $(\K^\alpha[G],\K^\alpha[G])$, with $\xi_g$ the left multiplication by $x_g$ defines a diagram of the form \eqref{eq:completion-B}. To see that it is universal, notice that for every $(B,N)$, the multiplication
\[
x_g\cdot n=\xi_g(n)
\]
makes $N$ a $(\K^\alpha[G],B)$-bimodule, and with this bimodule structure the isomorphism of right $B$-modules 
\[
\K^\alpha[G]\otimes_{\K^\alpha[G]}N\cong N
\]
provides the factorization \eqref{eq:kan-factorization}. Uniqueness is clear: if $M$ is a $(\K^\alpha[G],B)$-bimodule providing another factorization then we have an isomorphism of right $B$-modules $M\cong \K^\alpha[G]\otimes_{\K^\alpha[G]}M\cong N$ by definition of the factorization. Under this isomorphism, the left $\K^\alpha[G]$-module structure induced on $N$ becomes a left $\K^\alpha[G]$-module structure on $\K^\alpha[G]\otimes_{\K^\alpha[G]}M$ that is seen to be the natural left $\K^\alpha[G]$-module structure on $\K^\alpha[G]\otimes_{\K^\alpha[G]}M$ and so the left $\K^\alpha[G]$-module structure on $M$. This shows that $M$ and $N$ with the induced left $\K^\alpha[G]$-module structure are isomorphic as $(\K^\alpha[G],B)$-bimodules.
\end{proof}
\begin{corollary}\label{cor:kan}
    We have an equivalence { between the category
of lax homotopy commutative diagrams of the form
\[
 \begin{tikzcd}
            \mathbf{B}G
                \arrow[rr, "J_\alpha"]\arrow[dr]&{}&
                2\mathrm{Vect}_\K\\
                {}&\ast\arrow[u,shorten <=3pt,shorten > =1pt,Rightarrow,"N"',pos=.5]\ar[ur,"{B}"']&{}\\
            \end{tikzcd}
\]
\vskip -.8 cm
\noindent
and the category of             
              $(\K^\alpha[G],B)$-bimodules.}
\end{corollary}
{By taking $B=\mathbb{K}$ in the above Corollary, we get the following.}
\begin{corollary}
    We have an equivalence of categories
 {
 between the category of  projective representations of $G$ with 2-cocycle $\alpha$ and the category of $\K^\alpha[G]$-modules.}
\end{corollary}
\begin{remark}
Up to smallness issues, the construction presented in this section works for an arbitrary diagram of the form
\begin{equation*}
            \begin{tikzcd}
            \mathcal{C}
                \arrow[rr, "J"]\arrow[dr]&{}&
                2\mathrm{Vect}_\K\\
                {}&\ast&{}\\
            \end{tikzcd}.
                   \end{equation*}
To be rigorous, the construction in \cite{htt} only deals with $(\infty,1)$-categories, so it applies to the case of $\mathbf{B}G$ considered above but not directly to an arbitrary higher category $\mathcal{C}$. Here, since the target is $2\mathrm{Vect}_\K$, it is not restrictive to assume that $\mathcal{C}$ is a 2-category. In this case, Kan extensions of the kind we are considering here are implicitly considered in \cite{theo-reutter}.                   
We therefore see that there exists an algebra $\K^J[\mathcal{C}]$, unique up to Morita equivalence, such that there is an equivalence { between the category
of anomalous representations of $\mathcal{C}$ with anomaly $J$ and the category of $
\K^J[\mathcal{C}]\text{-modules}$.}
Notice that with this notation we have $\K^\alpha[G]\cong \K^{J_\alpha}[\mathbf{B}G]$.
\end{remark}

\end{document}